\documentclass[11pt,oneside,reqno]{amsart}
\usepackage[a4paper, total={450pt,675pt}]{geometry}
\usepackage[OT2,T1]{fontenc}
\usepackage[utf8]{inputenc}
\usepackage{amssymb}
\usepackage{bbm}
\usepackage{enumitem}
\usepackage[dvipsnames]{xcolor}
\usepackage[
    colorlinks=true,
    linkcolor=Maroon,
    citecolor=JungleGreen,
    urlcolor=NavyBlue]{hyperref}
\usepackage[capitalize,nameinlink]{cleveref}
\usepackage{dsfont}
\usepackage{tikz}
\usepackage{subcaption}
\usetikzlibrary{arrows}
\usepackage{soul}
\setstcolor{red}

\makeatletter
\def\@tocline#1#2#3#4#5#6#7{\relax
  \ifnum #1>\c@tocdepth 
  \else
    \par \addpenalty\@secpenalty\addvspace{#2}%
    \begingroup \hyphenpenalty\@M
    \@ifempty{#4}{%
      \@tempdima\csname r@tocindent\number#1\endcsname\relax
    }{%
      \@tempdima#4\relax
    }%
    \parindent\z@ \leftskip#3\relax \advance\leftskip\@tempdima\relax
    \rightskip\@pnumwidth plus4em \parfillskip-\@pnumwidth
    #5\leavevmode\hskip-\@tempdima
      \ifcase #1
       \or\or \hskip 2em \or \hskip 2em \else \hskip 3em \fi%
      #6\nobreak\relax
    \dotfill\hbox to\@pnumwidth{\@tocpagenum{#7}}\par
    \nobreak
    \endgroup
  \fi}
\makeatother

\usepackage[
    backend=biber,
    style=alphabetic,
    maxbibnames=10,
    sorting=nyt]{biblatex}

\DeclareLabelalphaTemplate{
  \labelelement{
    \field[final]{shorthand}
    \field{label}
    \field[strwidth=3]{labelname}}
  \labelelement{
    \field[strwidth=2,strside=right]{year}}
}
\DeclareFieldFormat[article,incollection]{title}{#1}
\DeclareFieldFormat{date}{#1}
\DeclareFieldFormat{pages}{{#1}}
\DeclareFieldFormat{doi}{\href{http://www.ams.org/mathscinet-getitem?mr=#1}{#1}}

\renewbibmacro{in:}{%
  \ifentrytype{article}{}{\printtext{\bibstring{in}\intitlepunct}}}

\newtheorem{theorem}{Theorem}[section]

\newtheorem{lemma}[theorem]{Lemma}

\newtheorem{corollary}[theorem]{Corollary}

\newtheorem{remark}[theorem]{Remark}

\crefname{section}{Sect.}{section}
\numberwithin{equation}{section}
\DeclareMathOperator{\Ad}{Ad}
\DeclareMathOperator{\im}{Im}
\DeclareMathOperator{\re}{Re}
\DeclareMathOperator{\supp}{supp}
\DeclareMathOperator{\sign}{sign}
\DeclareMathOperator{\const}{const.}

\begin{document}
\title[Schr\"odinger equation on non-compact symmetric spaces]
{Schr\"odinger equation on non-compact symmetric spaces}

\author
{Jean-Philippe Anker, Stefano Meda, Vittoria Pierfelice,\\
Maria Vallarino and Hong-Wei Zhang}

\begin{abstract}
We establish sharp-in-time kernel and dispersive estimates 
    for the Schrödinger equation on non-compact Riemannian symmetric spaces of any rank. 
    Due to the particular geometry at infinity and the Kunze-Stein phenomenon, these properties are more pronounced in large time 
    and enable us to prove the global-in-time Strichartz inequality for a larger
    family of admissible couples than in the Euclidean case.
    Consequently, we obtain the global well-posedness for the corresponding semilinear equation 
    with lower regularity data and some scattering properties for small powers
    which are known to fail in the Euclidean setting.
    The crucial kernel estimates are achieved by combining the stationary phase method based on a subtle barycentric decomposition, a subordination formula of the Schrödinger group to the wave propagator and an improved Hadamard parametrix.
\end{abstract}

\makeatletter
\@namedef{subjclassname@2020}{\textnormal{2020} \it{Mathematics Subject Classification}}
\makeatother
\subjclass[2020]
{22E30, 35J10, 35P25, 43A85, 43A90}

\keywords{
non-compact symmetric space,
semilinear Schr\"odinger equation, 
pointwise kernel estimates,
dispersive property,
Strichartz inequality, 
global well-posedness}

\maketitle
\tableofcontents

\section{Introduction}\label{Section.1 Intro}
\subsection{Strichartz inequality}
The Strichartz inequality, which has proved to be an efficient tool 
    in the study of the nonlinear Schrödinger equation,
    has been extensively investigated over the past five decades. 
    Let us start by recalling some well-known results in the Euclidean setting.
    Consider the free Schrödinger equation
    \begin{align}\label{S1 Schrodinger R}
    \begin{cases}
        i\partial_{t}u(t,x)\,+\,\Delta_{x}u(t,x)\,=\,0,
        \qquad\,t\in\mathbb{R},\,\,\,x\in\mathbb{R}^{d},\\[5pt]
        u(0,x)=f(x),
    \end{cases}
    \end{align}
    whose solution is given by the convolution of the initial data with the kernel of the Schrödinger
    propagator:
    \begin{align}\label{S1 convolution R}
        u(t,\,.\,)\,=\,
        e^{it\Delta}\,f\,=\,
        f*\big\lbrace{(4\pi{i}t)^{-\frac{d}{2}}\,e^{i\frac{|\,.\,|^{2}}{4t}}}\big\rbrace.
    \end{align}
    The Strichartz inequality refers to an a priori estimate of the solution to \eqref{S1 Schrodinger R}:
    in dimension $d\ge2$, if a couple $(p,q)$ is \textit{admissible}, in the sense that
    \begin{align}\label{S1 admissible R}
        \frac{2}{p}\,+\frac{d}{q}\,=\,\frac{d}{2}
        \qquad\textnormal{with}\,\,\,p,q\,\ge\,2\,\,\,
        \textnormal{and}\,\,\,(p,q,d)\,\neq\,(2,\infty,2),
    \end{align}
    then there exists a constant $C>0$ such that global solutions to \eqref{S1 Schrodinger R} satisfy
    \begin{align}\label{S1 Strichartz R}
    \|u\|_{L^{p}(\mathbb{R};\,L^{q}(\mathbb{R}^{d}))}\,
    \le\,C\,\|f\|_{L^{2}(\mathbb{R}^{d})}.
    \end{align}
Such an inequality means that, endowed with suitable space-time Lebesgue norms, 
    the solution to the Schrödinger equation can be estimated in terms of the $L^{2}$ norm of the initial data. 
    This type of estimate appeared in the pioneering 
    works \cite{Tom1975,Seg1976,Str1977} around 1976 and was fully developed after Keel-Tao 
    completed the proof for the endpoint in 1998, see \cite{KeTa1998}. 
    Over these twenty years, there have been  many remarkable contributions to this problem, 
    and we refer to \cite{GiVe1995,Caz2003,Tao2006} for more details.   

A related fundamental problem is to figure out how the underlying geometry affects the Strichartz inequality. 
    This helps us to understand the behavior of solutions to the partial differential equations in the non-Euclidean background.
    Let $\mathcal{M}$ be a Riemannian manifold of dimension $d\ge2$ and $\Delta$ be the 
    Laplace-Beltrami operator on $\mathcal{M}$. The question is, for which $(p,q)$ and $s\ge0$,
    the following inequality holds:
    \begin{align}\label{S1 Strichartz M}
    \|e^{it\Delta}f\|_{L^{p}(I;\,L^{q}(\mathcal{M}))}\,
    \le\,C\,\|f\|_{H^{s}(\mathcal{M})}.
    \end{align}
    We say that \eqref{S1 Strichartz M} holds global-in-time if $I=\mathbb{R}$.
    Here $H^{s}$ denotes the standard $L^{2}$ Sobolev spaces. 
    If \eqref{S1 Strichartz M} is available only for some $s>0$,
    we say that the Strichartz inequality holds \textit{with loss of $s$ derivatives}. 
    The inequality \eqref{S1 Strichartz R} shows that, when $\mathcal{M}=\mathbb{R}^{d}$,
    the global-in-time Strichartz inequality holds without loss
    for all $(p,q)$ satisfying the admissible condition \eqref{S1 admissible R}.
    However, this is not always the case on manifolds.

Apart from some \textit{non-trapping} manifolds
(see, for instance, \cite{StTa2002,HTW2005,BoTz2007}),
the Strichartz inequality is more or less understood 
on \textit{compact} manifolds.
In general, \eqref{S1 Strichartz M} cannot hold 
global-in-time in the compact case. 
    Bourgain proved that, on the flat torus $\mathcal{M}=\mathbb{T}^{d}$, 
    if a couple $(p,q)$ satisfies the admissible condition \eqref{S1 admissible R}, 
    then \eqref{S1 Strichartz M} holds for $I=\mathbb{T}$ and $s>\frac{d}{4}-\frac{1}{2}$,
    see \cite{Bou1993a,Bou1993b}.  Later, Burq, Gérard and Tzvetkov \cite{BGT2004} showed that,
    on arbitrary compact Riemannian manifolds without boundary, \eqref{S1 Strichartz M} holds 
    with loss of $s=1/p$ derivatives for any finite time interval $I$ and $(p,q)$
    satisfying \eqref{S1 admissible R}. This result is sharp in some particular cases,
    such as on a three-dimensional sphere, where a $1/2$ loss must occur in \eqref{S1 Strichartz M}.

In this paper we consider \textit{non-compact} Riemannian symmetric spaces 
    of \textit{any rank}, which form an important class of non-positively curved Riemannian manifolds. Due to their exponential volume growth at infinity and the validity of the Kunze-Stein phenomenon,
    one expects stronger dispersive phenomena than in the Euclidean setting,
    hence better Strichartz inequality, well-posedness and scattering results.
    This was indeed brought to light in the study of the Schrödinger equation on 
    real hyperbolic spaces, which are the simplest symmetric spaces of non-compact type 
    and \textit{rank one}, see for instance \cite{Pie2006,Ban2007,BCS2008,AnPi2009,IoSt2009}.
    See also \cite{APV2011} for similar results on Damek-Ricci spaces 
    (a class of harmonic manifolds which includes all non-compact symmetric spaces of rank one).
    Extending these rank one results to higher rank is a natural but challenging problem
    since the Schrödinger kernel is sensitive to the geometry of the underlying manifolds,
    and the spherical Fourier analysis is well-known to be much more intricate on the higher rank
    symmetric spaces. In the next subsections, we will explain in detail the difficulties involved 
    and share some new ideas to overcome them. 
    Roughly saying, we establish in the present paper a stronger global-in-time Strichartz inequality,
    in the sense that the related family of admissible pairs 
    is \textit{significantly larger} than the best possible in the Euclidean setting. 
    This is achieved by a subtle analysis of the Schrödinger kernel.
Once the Schrödinger equation is understood on globally symmetric spaces,
    studying it on locally symmetric spaces is a further natural problem.
    See \cite{BGH2010,FMM2018,Wan2019} for some first results in rank one.
    
\subsection{Statements of main results}
We adopt the standard $TT^{*}$ duality argument to establish the Strichartz inequality, see \cite{Kat1987,GiVe1995}. This argument relies on the dispersive estimates of the Schrödinger propagator, which can be easily obtained in $\mathbb{R}^{d}$ since the corresponding convolution kernel is explicitly defined. However, such a fundamental piece of information is lacking on a general manifold. Our primary task is to find sharp pointwise estimates of the Schrödinger kernel. 

Consider a non-compact symmetric space $\mathbb{X}=G/K$ of rank $\ell$, where $G$ and $K$ are suitable Lie groups. Let $d\ge2$ and $D\ge3$ be its manifold dimension and dimension at infinity (see \cref{Section.2 Prelim} for more details about these notations). We denote by $\Delta$ the Laplace-Beltrami operator on $\mathbb{X}$ and consider the free Schrödinger equation
    \begin{align}\label{S1 Schrodinger X}
    \begin{cases}
        i\partial_{t}u(t,x)\,+\,\Delta_{x}u(t,x)\,=\,0,
        \qquad\,t\in\mathbb{R},\,\,\,x\in\mathbb{X},\\[5pt]
        u(0,x)=f(x),
    \end{cases}
    \end{align}
whose solution is given by
\begin{align*}
    u(t,x)\,
    =\,e^{it\Delta}f(x)\,
    =\,f*s_{t}(x)\,
    =\,\int_{\mathbb{X}}\,dy\,f(y)\,s_{t}(y^{-1}x).
\end{align*}
Here $s_{t}$ is the bi-$K$-invariant convolution kernel of the Schrödinger propagator $e^{it\Delta}$. As mentioned above, except for a few particular cases, for instance when $G$ is complex (see \cite{Gan1968}), the kernel $s_{t}$ has no explicit expression as in \eqref{S1 convolution R}. Moreover, the expression obtained by using the inverse Abel transform on real hyperbolic spaces (see \cite[p.1859]{AnPi2009}) is also unavailable in higher rank. Our main result is the the following sharp-in-time kernel estimates on non-compact symmetric spaces of any rank.

\begin{theorem}
[Pointwise kernel estimates]\label{S1.Main Thm Kernel}
There exist \,$C\!>\!0$ and \,$N\!>\!0$
such that the fol\-low\-ing estimates hold,
for all $t\in\mathbb{R}^{*}$ and $x\in\mathbb{X}:$
\begin{align}\label{S1 kernel estimates}
    |s_{t}(x)|\, 
    \le\,C\,(1+|x^{+}|)^{N}\,
        e^{-\langle\rho,\,x^{+}\rangle}\,
    \begin{cases}
        |t|^{-\frac{d}{2}} 
        &\textnormal{if \,}0<|t|<1,\\[5pt]
        |t|^{-\frac{D}{2}} 
        &\textnormal{if \,}|t|\ge1,
    \end{cases}
\end{align}
where $x^{+}$ denotes the radial component of $x$ in the Cartan decomposition and $\rho$ is the half sum of all positive roots.
\end{theorem}

\begin{remark}
Notice that $\langle{\rho,x^{+}}\rangle>0$ as $\rho$ and $x^{+}$ are two $\ell$-dimensional vectors in the so-called positive Weyl chamber, see \cref{Section.2 Prelim} for more details. Hence the large polynomial factor $(1+|x^{+}|)^{N}$ is harmless for proving the dispersive estimates because of the exponential factor $e^{-\langle\rho,x^{+}\rangle}$.
\end{remark}

In contrast to the Euclidean setting, the kernel $s_{t}$ here (expressed as \eqref{S3 Inverse kernel} according to the inverse spherical Fourier transform) behaves differently for small and large times and satisfies no rescaling. Let us emphasize that the existing methods in rank one fail to produce desired estimates since the \textit{Plancherel density} $|\mathbf{c}(\lambda)|^{-2}$ occurring in \eqref{S3 Inverse kernel} is not always a differential symbol in general, hence the standard stationary phase method fails. This is in fact a well-known difficulty in the study of higher rank spherical harmonic analysis, see for instance \cite{BOS1995}. As for the recent works about the wave equation on symmetric spaces \cite{AnZh2020,Zha2021}, the Schrödinger propagator does not enjoy the finite propagation speed as the wave propagator, and a more involved analysis is required. However, by borrowing some ideas from previous works and by combining them cleverly, we can obtain \cref{S1.Main Thm Kernel} on general symmetric spaces of non-compact type, see \cref{Section.3 Kernel}.

Once we establish these key pointwise kernel estimates, the dispersive properties of $e^{it\Delta}$ follow from the Kunze-Stein phenomenon (instead of Young's convolution inequality) and interpolation. Then we deduce the Strichartz inequality by using the $TT^{*}$ argument. The proofs are not so different from the ones on real hyperbolic spaces, but we still include the details in \cref{Section.4 Dispersive Strichartz} for the reader's convenience.

\begin{theorem}
[Dispersive properties]\label{S1.Main Thm Dispersive}
Let $2\le{q},\widetilde{q}<+\infty$.
Then there exists a constant $C>0$ such that 
following estimates hold for all $t\in\mathbb{R}^{*}:$
\begin{align}\label{S1.Main Dispersive properties}
    \|e^{it\Delta}\|_{
    L^{\widetilde{q}'}(\mathbb{X})\,
    \rightarrow\,L^{q}(\mathbb{X})}
    \le C
    \begin{cases}
        |t|^{-\max(\frac{1}{2}-\frac{1}{q},
        \frac{1}{2}-\frac{1}{\widetilde{q}})\,d} 
        &\textnormal{if \,}0<|t|<1,\\[5pt]
        |t|^{-\frac{D}{2}} 
        &\textnormal{if \,}|t|\ge1,
    \end{cases}
    \end{align}
where $\widetilde{q}$ and $\widetilde{q}'$ are dual indices
in the sense that $\frac{1}{\widetilde{q}}+\frac{1}{\widetilde{q}'}=1$.
\end{theorem}

\begin{remark}
\cref{S1.Main Thm Dispersive} covers previously known results in rank $1$, where the dimension at infinity is $D=3$. Notice that the large time decay in \eqref{S1.Main Dispersive properties} becomes faster in higher rank, see \eqref{S2 Dimensions}. Notice also that the estimate \eqref{S1.Main Dispersive properties} is quite different from the one in the Euclidean setting where
\begin{align*}
    \|e^{it\Delta_{\mathbb{R}^{d}}}\|_{
    L^{q'}(\mathbb{R}^{d})\,
    \rightarrow\,L^{q}(\mathbb{R}^{d})}\,
    \le\,C\,|t|^{-(\frac{1}{2}-\frac{1}{q})d}
    \qquad\forall\,t\in\mathbb{R}^{*}.
\end{align*}
In particular, the large time decay in \eqref{S1.Main Dispersive properties} depends no more on the indices $q$ or $\widetilde{q}$. Such a particular phenomenon on non-compact symmetric spaces yields a stronger Strichartz inequality.
\end{remark}

\begin{theorem}
[Strichartz inequality]\label{S1.Main Thm Strichartz}
Let $(p,q)$ and $(\widetilde{p},\widetilde{q})$ be two admissible pairs 
corresponding to the triangle
\begin{align}\label{S1 admissible X}
    \Big\lbrace{
    \Big(\frac{1}{p},\frac{1}{q}\Big)
    \in\Big(0,\frac12\Big]\times\Big(0,\frac12\Big)
    \,\Big|\,
    \frac{2}{p}+\frac{d}{q}\ge\frac{d}{2}
    }\Big\rbrace
    \,\bigcup\,
    \Big\lbrace{
    \Big(0,\frac12\Big)
    }\Big\rbrace.
\end{align}
Then there exists a constant $C>0$ such that, for any bounded or unbounded $I\subseteq\mathbb{R}$,
the solution to the inhomogeneous Schrödinger equation
\begin{align}\label{S1 Schrodinger F}
    i\partial_{t}u(t,x)\,+\,\Delta_{x}u(t,x)\,
    =\,F(t,x),
    \qquad\,u(0,x)=f(x), 
\end{align}
satisfies
\begin{align}\label{S1 Strichart inequality}
    \|u\|_{L^{p}(I;L^{q}(\mathbb{X}))}\,
    \le\,
        C\,\big(
        \|f\|_{L^{2}(\mathbb{X})}
        \,+\,\|F\|_{L^{\widetilde{p}'}
        (I;L^{\widetilde{q}'}(\mathbb{X}))}
        \big).
\end{align}
\end{theorem}

\begin{figure}[b]
\centering
\begin{tikzpicture}[scale=0.47][line cap=round,line join=round,>=triangle 45,x=1.0cm,y=1.0cm]
\draw[->,color=black] (0.,0.) -- (8.805004895364174,0.);
\foreach \x in {,2.,4.,6.,8.,10.}
\draw[shift={(\x,0)},color=black] (0pt,-2pt);
\draw[color=black] (8,0.08808504628984362) node [anchor=south west] {\large $\frac{1}{p}$};
\draw[->,color=black] (0.,0.) -- (0.,8);
\foreach \y in {,2.,4.,6.,8.}
\draw[shift={(0,\y)},color=black] (-2pt,0pt);
\draw[color=black] (0.11010630786230378,8) node [anchor=west] {\large $\frac{1}{q}$};
\clip(-5,-3.5) rectangle (11.805004895364174,9.404473957611293);
\fill[line width=2.pt,color=red,fill=red,fill opacity=0.15000000596046448] (0.45855683542994374,5.928142080369191) -- (5.85410415683891,5.8847171522290775) -- (5.864960388873938,4.136863794589543) -- cycle;
\draw [line width=1.pt,dash pattern=on 5pt off 5pt,color=red] (0.,6.)-- (6.,6.);
\draw [line width=1.pt,color=red] (6.,6.)-- (6.,4.);
\draw [line width=1.pt,color=red] (6.,4.)-- (0.,6.);
\draw [line width=0.8pt] (6.,4.)-- (6.,0.);
\draw [->,line width=0.5pt,color=red] (1.8,-1.2) -- (2.891937454136701,4.811295161325333);
\begin{scriptsize}
\draw [fill=red] (0.,6.) circle (5pt);
\draw[color=black] (-0.5709441083587729,6) node {\large $\frac{1}{2}$};
\draw[color=black] (-1.5,4) node {\large $\frac{1}{2}- \frac{1}{d}$};
\draw[color=black] (-0.5,-0.2) node {\large $0$};
\draw[color=black] (6.2,-0.7) node {\large $\frac{1}{2}$};
\draw[color=red] (1.8,-2.5) node {\large $\frac{1}{p} = \frac{d}{2} \big(\frac{1}{2} - \frac{1}{q}\big)$};
\draw [color=red] (6.,6.) circle (5pt);
\draw [fill=red] (6.,4.) circle (5pt);
\end{scriptsize}
\end{tikzpicture}
\caption{Admissibility in dimension $d\ge3$.}
\label{S1 admissible figure}
\end{figure}
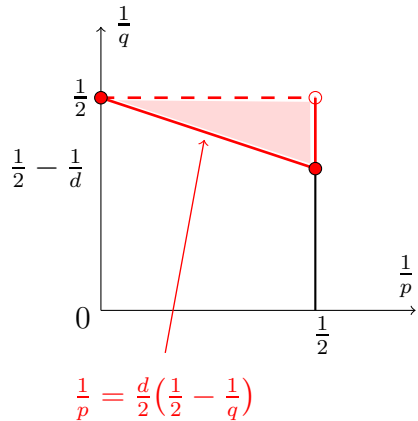

\begin{remark}
The inequality \eqref{S1 Strichart inequality} itself is analogous to the one in the Euclidean setting, but the admissible set \eqref{S1 admissible X} is much larger than \eqref{S1 admissible R}. The latter one corresponds only to the lower edge of the admissible triangle \eqref{S1 admissible X}, see \cref{S1 admissible figure}. This is due to the large-scale geometry of symmetric spaces, which yields better dispersive properties.
\end{remark}

\subsection{Well-posedness and scattering for the semilinear Schr\"odinger equation}\label{Section.5 NLS}
The \textit{fixed point argument}, which consists in finding a suitable contraction mapping defined on an appropriate Banach space, is nowadays a standard method to prove the \textit{well-posedness} of nonlinear partial differential equations. Due to the stronger Strichartz inequality proved in \cref{S1.Main Thm Strichartz}, we are able to obtain the following  better results on non-compact symmetric spaces. The proofs are adapted straightforwardly from the rank one case considered in \cite{AnPi2009,APV2011} and are therefore omitted.

Consider the semilinear Schr\"odinger equation:
    \begin{align}\tag{NLS}\label{S5 NLS}
    \begin{cases}
        i\partial_{t}u(t,x)\,+\,\Delta_{x}u(t,x)\,=\,F(u(t,x)),
        \qquad\,t\in\mathbb{R},\,\,\,x\in\mathbb{X},\\[5pt]
        u(0,x)=f(x).
    \end{cases}
    \end{align}
where $F$ is a \textit{power-like} nonlinearity 
of order $\gamma>1$ in the sense that 
\begin{align*}
    |F(u)|\,\le\,C\,|u|^{\gamma}
    \qquad\textnormal{and}\qquad
    |F(u)-F(v)|\,\le\,C\,
    (|u|^{\gamma-1}+|v|^{\gamma-1})\,|u-v|.
\end{align*}
Let $H^{1}(\mathbb{X})$ be the Sobolev space defined as the image of $L^{2}(\mathbb{X})$ under the operator $(-\Delta)^{-\frac12}$. We have the following well-posedness results for the semilinear Schrödinger equation \eqref{S5 NLS}:

\begin{itemize}[leftmargin=20pt]
    \item\vspace{5pt}
        If $1<\gamma\le1+\frac{4}{d}$, the Cauchy problem
        \eqref{S5 NLS} is globally well-posed for small
        $L^2$ data.
        
    \item\vspace{5pt}
        If $1<\gamma<1+\frac{4}{d}$, the Cauchy problem
        \eqref{S5 NLS} is locally well-posed for \textit{arbitrary}
        $L^2$ data. Moreover, if $F$ is in addition \textit{gauge invariant}, 
        namely, if $\im\lbrace{F(u)\bar{u}}\rbrace=0$,
        then the $L^2$ conservation of mass 
        implies the global well-posedness for arbitrary
        $L^2$ data in this subcritical case.
    
    \item\vspace{5pt}
        If $1<\gamma\le1+\frac{4}{d-2}$, the Cauchy problem
        \eqref{S5 NLS} is globally well-posed for small
        $H^1$ data.
 
    \item\vspace{5pt}
        If $1<\gamma<1+\frac{4}{d-2}$, the Cauchy problem
        \eqref{S5 NLS} is locally well-posed for \textit{arbitrary}
        $H^1$ data. Moreover, if $F$ is in addition \textit{defocusing},
        namely, if there exists a nonnegative $\mathcal{C}^{1}$ function
            $G$ such that $F(u)=G'(|u|^{2})u$,
        then the $H^1$ conservation of energy implies the 
        global well-posedness for arbitrary $H^1$ data 
        in this subcritical case.
\end{itemize}

Notice that these results are better than the known ones on Euclidean spaces.
For instance, global well-posedness for small $L^{2}$ initial data holds
for any exponent $1<\gamma\le1+\frac{4}{d}$ on $\mathbb{X}$,
while on $\mathbb{R}^{d}$ one must assume in addition gauge
invariance. 
However, under this condition, one can handle arbitrary $L^2$ 
data by using conservation laws.

The Strichartz inequality can also be used to prove \textit{scattering} results, which means that global solutions to the \textit{nonlinear} Schrödinger equation behave asymptotically as solutions to the \textit{linear} equation. Specifically, \cref{S1.Main Thm Strichartz} implies the following scattering results.

\begin{itemize}[leftmargin=20pt]
    \item
        If $1<\gamma\le1+\frac{4}{d}$, then the global
        solution $u(t,x)$ to the Cauchy problem \eqref{S5 NLS}
        corresponding to small $L^2$ data satisfies the following
        scattering property: there exists $u_{\pm}\in{L}^{2}(\mathbb{X})$
        such that
        \begin{align*}
           \|u(t,\cdot)-e^{it\Delta}u_{\pm}\|_{L^{2}(\mathbb{X})}
           \,\longrightarrow\,0
           \qquad\textnormal{as}\qquad
           t\,\longrightarrow\,\pm\infty.
        \end{align*}
    
    \item
        If $1<\gamma\le1+\frac{4}{d-2}$, then the global
        solution $u(t,x)$ to the Cauchy problem \eqref{S5 NLS}
        corresponding to small $H^1$ data satisfies the following
        scattering property: there exists $u_{\pm}\in{H}^{1}(\mathbb{X})$
        such that
        \begin{align*}
           \|u(t,\cdot)-e^{it\Delta}u_{\pm}\|_{H^{1}(\mathbb{X})}
           \,\longrightarrow\,0
           \qquad\textnormal{as}\qquad
           t\,\longrightarrow\,\pm\infty.
        \end{align*}
        This scattering property remains valid for \textit{any} $H^{1}$ initial data in the subcritical case $1<\gamma<1+\frac{4}{d-2}$ if the nonlinearity $F$ is assumed to be defocusing.
\end{itemize}

Notice that, on $\mathbb{R}^{d}$, these scattering properties are known to fail for small powers $\gamma\in(1,1+\frac{2}{d}]$. However, our stronger Strichartz inequality on non-compact symmetric spaces makes these properties possible in the full range.

\subsection{Layout}

After reviewing spherical Fourier analysis on non-compact
symmetric spaces in \cref{Section.2 Prelim}, 
we prove in \cref{Section.3 Kernel} our main result,
namely the pointwise kernel estimates.
More precisely, we start with the large time estimate, which is surprisingly straightforward, see
\cref{S3 subsec large}. In small time, we get the desired estimates by combining the stationary phase method based on the barycentric decomposition in \cref{S2 barycentric dichotomy}, the subordination formula \eqref{S3.2 Kernel (wave)}, and the improved Hadamard parametrix \eqref{S3.2 Parametrix}, see \cref{S3 subsec small}.
By adapting the method carried out in rank one, we deduce 
in \cref{Section.4 Dispersive Strichartz} our stronger
dispersive and Strichartz inequalities.


\section{Preliminaries}\label{Section.2 Prelim}

In this section we review briefly harmonic analysis on Riemannian
symmetric spaces of non-compact type. We adopt the standard notation 
and refer to \cite{Hel1978, Hel2000,GaVa1988} for more details. 
Throughout this paper, the symbol $A\lesssim{B}$ between two positive expressions means that there is a constant $C>0$ such that $A\le CB$. The symbol $A\asymp B$ means that $A\lesssim B$ and $B\lesssim A$.

\subsection{Non-compact symmetric spaces}
Let $G$ be a semisimple Lie group, connected, non-compact,
with finite center, and $K$ be a maximal compact subgroup
of $G$. The homogeneous space $\mathbb{X}=G/K$ is 
a Riemannian symmetric space of non-compact type.
Let $\mathfrak{g}=\mathfrak{k}\oplus\mathfrak{p}$ be
the Cartan decomposition of the Lie algebra of $G$,
the Killing form of $\mathfrak{g}$ induces a $K$-invariant inner product 
$\langle.\,,\,.\rangle$
on $\mathfrak{p}$, hence a $G$-invariant
Riemannian metric on $\mathbb{X}$.
Fix a maximal abelian subspace $\mathfrak{a}$ in
$\mathfrak{p}$. The rank of $\mathbb{X}$ is the dimension
$\ell$ of $\mathfrak{a}$.
We identify $\mathfrak{a}$ with its dual $\mathfrak{a}^{*}$
by means of the inner product inherited from $\mathfrak{p}$.
Let $\Sigma\subset\mathfrak{a}$
be the root system of $( \mathfrak{g},\mathfrak{a})$ and
denote by $W$ the Weyl group associated to $\Sigma$. 
Once a positive Weyl chamber 
$\mathfrak{a}^{+}\subset\mathfrak{a}$ has been selected,
$\Sigma^{+}$ (resp. $\Sigma_{r}^{+}$ or $\Sigma_{s}^{+}$)
denotes the corresponding set of positive roots 
(resp. positive reduced roots or simple roots).
Let $d$ be the dimension of $\mathbb{X}$ and let $D$ be the
so-called dimension at infinity or pseudo-dimension of $\mathbb{X}$: 
\begin{align}\label{S2 Dimensions}
\textstyle
    d\,=\,
    \ell\,+\,\sum_{\alpha \in \Sigma^{+}}\,m_{\alpha}
    \quad\textnormal{and}\quad
    D\,=\,\ell\,+\,2|\Sigma_{r}^{+}|,
\end{align}
where $m_{\alpha}$ is the dimension of the positive 
root subspace $\mathfrak{g}_{\alpha}$.
Notice that these two dimensions behave quite differently.
For example, $D=3$ while $d\ge2$ is arbitrary in rank one,
$D=d$ if $G$ is complex,
and $D>d$ (actually $D=2d-\ell$) if $G$ is split.
\vspace{5pt}

Let $\mathfrak{n}$ be the nilpotent Lie subalgebra 
of $\mathfrak{g}$ associated to $\Sigma^{+}$ 
and let $N = \exp \mathfrak{n}$ be the corresponding 
Lie subgroup of $G$. We have the decompositions 
\begin{align*}
    \begin{cases}
        \,G\,=\,N\,(\exp\mathfrak{a})\,K 
        \quad&\textnormal{(Iwasawa)}, \\
        \,G\,=\,K\,(\exp\overline{\mathfrak{a}^{+}})\,K
        \quad&\textnormal{(Cartan)}.
    \end{cases}
\end{align*}
In the Cartan decomposition, the Haar measure on $G$ writes
\begin{align*}
    \int_{G}\,dx\,f(x)\,
    =\,\const\,\int_{K}\,dk_{1}\,
    \int_{\mathfrak{a}^{+}}\,dx^{+}\,\delta(x^{+})\, 
    \int_{K}\,dk_{2}\,f(k_{1}(\exp x^{+})k_{2})\,,
\end{align*}
with density
\begin{align}\label{S2 density delta}
    \delta(x^{+})\,
    =\,\prod_{\alpha\in\Sigma^{+}}\,
        (\sinh\langle{\alpha,x^{+}}\rangle)^{m_{\alpha}}\,
    \asymp\,
        \prod_{\alpha\in\Sigma^{+}}
        \Big\lbrace 
        \frac{\langle\alpha,x^{+}\rangle}
        {1+\langle\alpha,x^{+}\rangle}
        \Big\rbrace^{m_{\alpha}}\,
        e^{2\langle\rho,x^{+}\rangle}
    \quad\forall\,x^{+}\in\overline{\mathfrak{a}^{+}}. 
\end{align}
Here $\rho\in\mathfrak{a}^{+}$ denotes the half sum of all positive roots $\alpha\in\Sigma^{+}$ counted with their multiplicities $m_{\alpha}$:
\begin{align*}
    \rho\,
    =\,\frac{1}{2}\,\sum_{\alpha\in\Sigma^{+}} 
        \,m_{\alpha}\,\alpha.
\end{align*}

\subsection{Spherical Fourier analysis on symmetric spaces}
Let $\mathcal{S}(K \backslash{G}/K)$ be the Schwartz space 
of bi-$K$-invariant functions on $G$. 
The spherical Fourier transform $\mathcal{H}$ is defined by
\begin{align*}
    \mathcal{H}f(\lambda)\,
    =\,\int_{G}\,dx\,\varphi_{-\lambda}(x)\,f(x) 
    \quad\forall\,\lambda\in\mathfrak{a},\ 
    \forall\,f\in\mathcal{S}(K\backslash{G/K}),
\end{align*}
where $\varphi_{\lambda}\in\mathcal{C}^{\infty}
(K \backslash{G/K})$ denotes the spherical function 
of index $\lambda \in \mathfrak{a}$, 
which is a smooth bi-$K$-invariant eigenfunction 
of all invariant differential operators on $\mathbb{X}$,
in particular of the Laplace-Beltrami operator:
\begin{equation*}
    -\Delta\varphi_{\lambda}(x)\,
    =\,(|\lambda|^{2}+|\rho|^2)\,
    \varphi_{\lambda}(x).
\end{equation*}
The spherical functions have the following integral representation
\begin{align}\label{S2 Spherical Function}
    \varphi_{\lambda}(x) 
    = \int_{K}dk\ 
        e^{\langle{i\lambda+\rho,A(kx)}\rangle}
    \quad\forall\,
        \lambda\in{\mathfrak{a}},
\end{align}
where $A(kx)$ denotes the $\mathfrak{a}$-component 
in the Iwasawa decomposition of $kx$. 
They satisfy the basic estimate
\begin{align*}
    |\varphi_{\lambda}(x)| \,
    \le\,\varphi_{0}(x)
    \quad\forall\,\lambda\in\mathfrak{a},\
    \forall\,x\in\,G,
\end{align*}
where
\begin{align}\label{S2 phi0}
    \varphi_{0}(\exp{x^{+}})\,
    \asymp\,\Big\lbrace 
        \prod_{\alpha\in\Sigma_{r}^{+}} 
        (1+\langle\alpha,x^{+}\rangle)
        \Big\rbrace\,
        e^{-\langle\rho, x^{+}\rangle} 
        \quad\forall\,x\in G.
\end{align}
Denote by $\mathcal{S}(\mathfrak{a})^{W}$ the subspace 
of $W$-invariant functions in the Schwartz space
$\mathcal{S}(\mathfrak{a})$. Then $\mathcal{H}$ is an
isomorphism between $\mathcal{S}(K\backslash{G/K})$ 
and $\mathcal{S}(\mathfrak{a})^{W}$. 
The inverse spherical Fourier transform is given by
\begin{align}\label{S2 Inverse Fourier}
    f(x)\,
    =\,C_0\,\int_{\mathfrak{a}}\,d\lambda\,
        |\mathbf{c(\lambda)}|^{-2}\,
        \varphi_{\lambda}(x)\,
        \mathcal{H}f(\lambda) 
    \quad \forall\,x\in{G},\ 
        \forall\,f\in\mathcal{S}(\mathfrak{a})^{W},
\end{align}
where $C_0 > 0$ is a constant depending only 
on the geometry of $\mathbb{X}$.
By using the Gindikin-Karpelevič formula of 
the Harish-Chandra $\mathbf{c}$-function 
(see \cite{Hel2000} or \cite{GaVa1988}), 
we can write the Plancherel density as
\begin{align}\label{S2 Plancherel Density}
    |\mathbf{c}(\lambda)|^{-2}\,
    =\,\prod_{\alpha\in\Sigma_{r}^{+}}\,
        |\mathbf{c}_{\alpha}
        (\langle\alpha,\lambda\rangle)|^{-2},
\end{align}
with
\vspace{-4mm}
\begin{align*}
\textstyle
    c_{\alpha}(v)\,=\,\overbrace{
    \tfrac{\Gamma(\frac{\langle{\alpha,\rho}\rangle}
        {\langle{\alpha,\alpha}\rangle}
        +\frac{1}{2} m_{\alpha})}
        {\Gamma(\frac{\langle{\alpha,\rho}\rangle}
        {\langle{\alpha,\alpha}\rangle})}\,
    \tfrac{\Gamma(\frac{1}{2}
        \frac{\langle{\alpha,\rho}\rangle}
        {\langle{\alpha,\alpha}\rangle} 
        +\frac{1}{4} m _{\alpha} 
        + \frac{1}{2} m_{2\alpha})}
        {\Gamma(\frac{1}{2}
        \frac{\langle{\alpha,\rho}\rangle}
        {\langle {\alpha,\alpha}\rangle} 
        + \frac{1}{4} m_{\alpha})}}^{C_\alpha}\,
    \frac{\Gamma(iv)}
        {\Gamma(iv+ \frac{1}{2}m_{\alpha})}\,
    \frac{\Gamma(\frac{i}{2}v 
        + \frac{1}{4} m_{\alpha})}
        {\Gamma(\frac{i}{2}v 
        +\frac{1}{4} m_{\alpha} 
        + \frac{1}{2} m_{2\alpha})}.
\end{align*}
Notice that $|\mathbf{c}_{\alpha}|^{-2}$ is an inhomogeneous
differential symbol on $\mathbb{R}$ of order $m_{\alpha}+m_{2\alpha}$,
for every $\alpha\in\Sigma_{r}^{+}$. Hence $|\mathbf{c}(\lambda)|^{-2}$ 
is a product of one-dimensional symbols, but not a symbol on 
$\mathfrak{a}$ in general. 
It satisfies indeed
\begin{align}\label{S2 Density Estimates}
    \begin{cases}
        |\mathbf{c}(\lambda)|^{-2}\,
        \lesssim\,
        |\lambda|^{D-\ell} 
        & \textnormal{if \,} |\lambda|\le1,\\[5pt]
        |\nabla_{\mathfrak{a}}^{k}\mathbf{c}(\lambda)|^{-2}\,
        \lesssim\,
        |\lambda|^{d-\ell} 
        & \textnormal{if \,} |\lambda|\ge1,
    \end{cases}
\end{align}
for all $k\in\mathbb{N}$. The estimate away from the origin means that one cannot obtain any additional decay from general derivatives, hence the standard stationary phase method fails. To overcome this difficulty, we introduce the following barycentric decomposition.
\begin{theorem}\label{S2 barycentric dichotomy}
There exists a smooth partition of unity 
\begin{align}\label{S2 Partition}
    \sum_{w\in W}\,\sum_{1\le{j}\le\ell}
    \chi_{w.S_{j}}\,=\,1
    \qquad\textnormal{on}\qquad
    \mathfrak{a}\smallsetminus\lbrace0\rbrace
\end{align}
consisting of homogeneous symbols $\chi_{w.S_{j}}$ of order $0$, and some nonzero vectors $w.\Lambda_{j}\in\supp\chi_{w.S_{j}}$ such that for all $w\in{W}$ and $1\le{j}\le\ell$,
\begin{align}\label{S2 barycentric 1}
    |\langle{w.\Lambda_{j},\lambda}\rangle|\asymp|\lambda|
    \qquad\,\forall\,\lambda\in\supp\chi_{w.S_{j}}
\end{align}
and the following dichotomy holds: for every $\alpha\in\Sigma$, 
\begin{itemize}[leftmargin=20pt]
    \vspace{5pt}\item
        either $\langle\alpha,w.\Lambda_{j}\rangle=0$,
    \vspace{5pt}\item 
        or $|\langle\alpha,\lambda\rangle|\asymp|\lambda|$ for all $\lambda\in\supp\chi_{w.S_{j}}$.
\end{itemize}
\end{theorem}
This theorem allows us to split up the Cartan subspace $\mathfrak{a}$ into several subcones $\supp\chi_{w.S_{j}}$ by using some appropriate cut-off functions $\chi_{w.S_{j}}$. In each subcone, we can choose a vector $w.\Lambda_{j}$ such that if a root $\alpha$ is not orthogonal to $w.\Lambda_{j}$, then the inner product between $\alpha$ and any vector $\lambda$ in this subcone is comparable to $|\lambda|$, see \cite[Subsection 2.2]{AnZh2020} for more details. Thanks to this decomposition, the Plancherel density can be handled as if it were a differential symbol, 
provided that we differentiate it in $\supp\chi_{w.S_{j}}$ along the well-chosen vector $w.\Lambda_{j}$.

\begin{corollary}\label{S2 dichotomy coro}
For every $w\in{W}$, $1\le{j}\le\ell$ and $k\in\mathbb{N}$, we have
\begin{align}\label{S2 diff Plancherel}
    \partial_{w.\Lambda_{j}}^{k}|\mathbf{c}(\lambda)|^{-2}\,
    =\,\mathrm{O}(|\lambda|^{d-\ell-k})
    \qquad\,\forall\lambda\in\supp\chi_{w.S_{j}}.
\end{align}
\end{corollary}

\begin{proof}
Let $\alpha$ be an arbitrary root in $\Sigma$. If $\langle\alpha,w.\Lambda_{j}\rangle=0$, then all terms
\begin{align*}
    \partial_{w.\Lambda_{j}}^{k}\,
    |\mathbf{c}_{\alpha}
    (\langle\alpha,\lambda\rangle)|^{-2}
    \qquad\forall\,k\in\mathbb{N}^{*}
\end{align*}
vanish. Otherwise, we deduce from the previous dichotomy that 
\begin{align*}
    \big| \partial_{w.\Lambda_{j}}^{k}\,
    |\mathbf{c}_{\alpha}
    (\langle\alpha,\lambda\rangle)|^{-2} \big|\,
    \lesssim\,
        (1+|\langle\alpha,\lambda\rangle
        |)^{m_{\alpha}+m_{2\alpha}-k}\,
    \asymp\,
        (1+|\lambda|)^{m_{\alpha}+m_{2\alpha}-k}
\end{align*}
for all $\lambda\in\supp\chi_{w.S_{j}}$, since $|\mathbf{c}_{\alpha}|^{-2}$ is an inhomogeneous differential symbol of order $m_{\alpha}+m_{2\alpha}$. We conclude by using \eqref{S2 Dimensions} and \eqref{S2 Plancherel Density}.
\end{proof}

\section{Pointwise estimates of the Schr\"odinger kernel}
\label{Section.3 Kernel}

For simplicity, we consider in this section the shifted
Schr\"odinger propagator $e^{-it\mathbf{D}^{2}}$ with 
$\mathbf{D}=\sqrt{-\Delta-|\rho|^2}$, and denote  by 
$s_{t}$ its bi-$K$-invariant convolution kernel. 
By using the inverse formula of the spherical Fourier transform,
we have
\begin{align}\label{S3 Inverse kernel}
    s_{t}(x)\,
    =\,C_0\,\int_{\mathfrak{a}}\,d\lambda\,
        |\mathbf{c}(\lambda)|^{-2}\,
        \varphi_{\lambda}(x)\,
        e^{-it|\lambda|^{2}}
    \qquad\forall\,t\in\mathbb{R}^{*},\,
    \forall\,x\in\mathbb{X}.
\end{align}
As usual, such an oscillatory integral makes sense
by applying standard procedures 
(as a limit of convergent integrals and/or
after performing several integrals by parts).
We will study \eqref{S3 Inverse kernel} differently,
depending whether $|t|$ is large or small.
Let us begin with the easier case where $|t|$ is large.

\subsection{Large time kernel estimate}\label{S3 subsec large}
Assume that $|t|\ge1$. In this case, we establish the 
following pointwise kernel estimate, by using the 
standard stationary phase method based on the elementary
estimate \eqref{S2 Density Estimates}. 
\begin{theorem}\label{S3.1 Kernel1}
There exist an integer $N>\max\lbrace{d,D}\rbrace$ and a constant $C>0$ such that
\begin{align*}
    |s_{t}(x)|\,
    \le\,C\,
        |t|^{-\frac{D}{2}}\,
        {\textstyle
        \big(1+|x|\big)^{N}}\,
        \varphi_{0}(x)
        \qquad\forall\,|t|\ge1,\,
        \forall\,x\in\mathbb{X}.
\end{align*}
\end{theorem}

\begin{proof}
By using the integral expression \eqref{S2 Spherical Function}
of the spherical function, we write
\begin{align}\label{S3.1 Kernel}
    s_{t}(x)\,
    =\,C_0\,
        \int_{K}\,dk\,e^{\langle{\rho,A(kx)}\rangle}\,
        \underbrace{
        \int_{\mathfrak{a}}\,d\lambda\,
        |\mathbf{c}(\lambda)|^{-2}\,
        e^{-it|\lambda|^{2}}\,
        e^{i\langle{\lambda,A(kx)}\rangle}
        }_{I(t,A(kx))}
\end{align}
where $A(kx)$ denotes the $\mathfrak{a}$-component 
in the Iwasawa decomposition of $kx$, which
satisfies $|A(kx)|\le|x|$ and which we abbreviate by
$A$ in the sequel.
\cref{S3.1 Kernel1} will follow from the estimate
\begin{align}\label{S3.1 Estimate I}
    |I(t,A)|\,
    \lesssim\,|t|^{-\frac{D}{2}}\,
    {\textstyle
        \big(1+|A|\big)^{N}}.
\end{align}
Let us split up
\begin{align*}
    I(t,A)\,
    =\,I_{0}(t,A)\,+\,I_{\infty}(t,A)\,
    =\,\int_{\mathfrak{a}}\,d\lambda\,
        \chi_{t}^{0}(\lambda)\,\dots\,
        +\,
        \int_{\mathfrak{a}}\,d\lambda\,
        \chi_{t}^{\infty}(\lambda)\,\dots
\end{align*}
where $\chi_{t}^{0}(\lambda)=\chi(\sqrt{|t|}|\lambda|)$
is a radial cut-off function such that
$\supp\chi_{t}^{0}\subset\,B(0,2|t|^{-\frac12})$,
$\chi_{t}^{0}=1$ on $B(0,|t|^{-\frac12})$
and $\chi_{t}^{\infty}=1-\chi_{t}^{0}$.
On the one hand, by using \eqref{S2 Density Estimates},
we easily estimate
\begin{align}\label{S3.1 Estimate I0}
    |I_{0}(t,A)|\,
    \lesssim\,
        \int_{|\lambda|\lesssim\,|t|^{-\frac12}}\,d\lambda\,
        |\mathbf{c}(\lambda)|^{-2}\,
    \lesssim\,
        |t|^{-\frac{D}{2}}.
\end{align}
On the other hand, after performing $N$ integrations by
parts based on
\begin{align}\label{S3.1 Phase}
\textstyle
    e^{-it|\lambda|^2}\,
    =\,-\frac{1}{2it}\,
        \sum_{j=1}^{\ell}\,\frac{\lambda_{j}}{|\lambda|^2}\,
        \frac{\partial}{\partial\lambda_{j}}\,
        e^{-it|\lambda|^2},
\end{align}
we obtain
\begin{align}\label{S3.1 I infinity}
    I_{\infty}(t,A)\,
    =\,(2it)^{-N}\,
        \int_{\mathfrak{a}}\,d\lambda\,
        e^{-it|\lambda|^{2}}\,
        \Big\lbrace{\textstyle
        \sum_{j=1}^{\ell}\,
        \frac{\partial}{\partial\lambda_{j}}\,
        \circ\,\frac{\lambda_{j}}{|\lambda|^2}
        }\Big\rbrace^{N}
        \Big\lbrace{
        \chi_{t}^{\infty}(\lambda)\,
        |\mathbf{c}(\lambda)|^{-2}\,
        e^{i\langle{\lambda,A}\rangle}
        }\Big\rbrace.
\end{align}
Assume that 
\begin{itemize}[leftmargin=20pt]
    \item 
        $N_{0}$ derivatives are applied to the cut-off
        function $\chi_{t}^{0}(\lambda)$,
        which produces $\mathrm{O}(|t|^{\frac{N_{0}}{2}})$,
        
    \item\vspace{5pt}
        $N_{1}$ derivatives are applied to the factors 
        $\frac{\lambda_{j}}{|\lambda|^2}$, which produces
        $\mathrm{O}(|\lambda|^{-N-N_{1}})$,

    \item\vspace{5pt}
        $N_{2}$ derivatives are applied to the Plancherel density $|\mathbf{c}(\lambda)|^{-2}$,
        which is not a symbol in general and which produces
        \begin{align*}
            \begin{cases}
            \,\mathrm{O}(|\lambda|^{D-\ell}) 
            &\textnormal{if \,}|\lambda|\le1,\\[5pt]
            \,\mathrm{O}(|\lambda|^{d-\ell})
            &\textnormal{if \,}|\lambda|\ge1,
            \end{cases}
        \end{align*}
    \item
        $N_{3}$ derivatives are applied to the exponential factor $e^{i\langle{\lambda,A}\rangle}$,
        which produces $\mathrm{O}(|A|^{N_3})$,
        \vspace{5pt}
\end{itemize}
with $N_{0}+N_{1}+N_{2}+N_{3}=N$.
If some derivatives hit the cut-off function $\chi_{t}^{0}(\lambda)$,
i.e., if $N_{0}\ge1$, then the integral reduces to a spherical shell where
$|\lambda|\asymp|t|^{-\frac12}$, and the contribution 
to \eqref{S3.1 I infinity} is estimated by
\begin{align}\label{S3.1 Estimate N0>1}
    |t|^{-\frac{N}{2}}\,
    |t|^{\frac{N_{0}}{2}}\,|t|^{\frac{N_{1}}{2}}\,
    |t|^{-\frac{D}{2}}\,|A|^{N_{3}}\,
    \lesssim\,
     |t|^{-\frac{D}{2}}\,{\textstyle\big(1+|A|\big)^{N}},
\end{align}
since $|t|\ge1$.
If $N_{0}=0$, then
\begin{align}\label{S3.1 Estimate N0=0}
    |I_{\infty}(t,A)|\,
    &\lesssim\,|t|^{-N}\,
    \int_{|\lambda|\gtrsim|t|^{-\frac12}}\,d\lambda\,
    \big|\nabla_{\lambda}^{N_{2}}\,
    |\mathbf{c}(\lambda)|^{-2}\big|\,
    |\lambda|^{-N-N_{1}}\,|A|^{N_{3}}\nonumber\\[5pt]
    &\lesssim\,|t|^{-N}\,|A|^{N_{3}}\,
    \Big\lbrace{
    \int_{|t|^{-\frac12}\lesssim|\lambda|\le1}\,
    d\lambda\,|\lambda|^{D-\ell-N-N_{1}}\,
    +\,
    \int_{|\lambda|\ge1}\,d\lambda\,
    |\lambda|^{d-\ell-N-N_{1}}\,
    }\Big\rbrace\nonumber\\[5pt]
    &\lesssim\,
    |t|^{-\frac{D}{2}}\,(1+|A|)^{N}
    +
    |t|^{-N}\,(1+|A|)^{N}\,
    \lesssim\,
    |t|^{-\frac{D}{2}}\,(1+|A|)^{N}
\end{align}
provided that $N>d$ and $N\ge\frac{D}{2}$.
In conclusion, \eqref{S3.1 Estimate I} follows
from \eqref{S3.1 Estimate I0}, \eqref{S3.1 Estimate N0>1} 
and \eqref{S3.1 Estimate N0=0}. 
\end{proof}

\begin{remark}
The analysis carried out in the proof of \cref{S3.1 Kernel1}
yields at best the following small time estimate
\begin{align}\label{S3.1 Bad estimate}
    |s_t(x)|\,
    \lesssim\,|t|^{-d}\,(1+|x|)^{d}\,
    \varphi_{0}(x)
    \qquad\forall\,0<|t|<1,\,
    \forall\,x\in\mathbb{X}.
\end{align}
\end{remark}

\subsection{Small time kernel estimate}\label{S3 subsec small}
Assume that $0<|t|<1$.
Our aim is to reduce the negative power $|t|^{-d}$
in \eqref{S3.1 Bad estimate} to $|t|^{-\frac{d}{2}}$.
We shall use different tools, depending on the size 
of $\frac{|x|}{\sqrt{|t|}}$.
If $\frac{|x|}{\sqrt{|t|}}$ is small, we decompose the 
Weyl chamber into several subcones according to the
barycentric decompositions described in 
\cref{S2 barycentric dichotomy}, and perform in each subcone 
several integrations by parts along a well chosen direction.
If $\frac{|x|}{\sqrt{|t|}}$ is large, we express
in addition the Schr\"odinger propagator in terms of 
the wave propagator 
and use the Hadamard parametrix.

\begin{theorem}\label{S3.2 Kernel2}
The following estimate holds, for $0<|t|<1$ 
and $|x|\le\sqrt{|t|}$:
\begin{align*}
    |s_{t}(x)|\,
    \lesssim\,
        |t|^{-\frac{d}{2}}\,
        \varphi_{0}(x).
\end{align*}
\end{theorem}

\begin{proof}
By resuming the notation in the proof 
of \cref{S3.1 Kernel1}, we have
\begin{align*}
    s_{t}(x)\,
    =\,C_0\,
        \int_{K}\,dk\,e^{\langle{\rho,A(kx)}\rangle}\,
        \big(I_{0}(t,A)+I_{\infty}(t,A)\big).
\end{align*}
Clearly,
\begin{align}\label{S3.2 Estimate I0}
    |I_{0}(t,A)|\,
    =\,\Big|\int_{\mathfrak{a}}\,d\lambda\,
        \chi_{t}^{0}(\lambda)\,
        |\mathbf{c}(\lambda)|^{-2}\,
        e^{-it|\lambda|^{2}}\,
        e^{i\langle{\lambda,A}\rangle}\Big|
    \lesssim
        \int_{|\lambda|\lesssim\,|t|^{-\frac12}}\,d\lambda\,
        |\mathbf{c}(\lambda)|^{-2}\,
    \lesssim\,
        |t|^{-\frac{d}{2}}.
\end{align}
In order to estimate
\begin{align*}
    I_{\infty}(t,A)\,
    =\,\int_{\mathfrak{a}}\,d\lambda\,
        \chi_{t}^{\infty}(\lambda)\,
        |\mathbf{c}(\lambda)|^{-2}\,
        e^{-it|\lambda|^{2}}\,
        e^{i\langle{\lambda,A}\rangle},
\end{align*}
we split up
\begin{align*}
    I_{\infty}(t,A)\,
    =\,\sum_{w\in W}\,\sum_{1\le j\le\ell}\,
    \underbrace{
    \int_{\mathfrak{a}}\,d\lambda\,\chi_{w.S_{j}}(\lambda)\,
    \chi_{t}^{\infty}(\lambda)\,|\mathbf{c}(\lambda)|^{-2}\,
    e^{-it|\lambda|^{2}}\,e^{i\langle\lambda,A\rangle}
    }_{I_{w.S_{j}}(t,x)}.
\end{align*}
according to the barycentric decomposition 
\eqref{S2 Partition}.
Next, we study $I_{w.S_{j}}(t,x)$ by performing
$N$ integrations by parts based on
\begin{align*}
        e^{-it|\lambda|^{2}}\,=\,
        -\,\tfrac{1}{2it}\,
        \tfrac{1}{\langle{w.\Lambda_{j},\lambda}\rangle}\,
        \partial_{w.\Lambda_{j}}\,e^{-it|\lambda|^{2}},
    \end{align*}
which yields
\begin{align*}
    I_{w.S_{j}}(t,x)\,=\,(2it)^{-N}\,
    \int_{\mathfrak{a}}\,d\lambda\,e^{-it|\lambda|^{2}}
    \Big\lbrace
    \partial_{w.\Lambda_{j}}\circ
    \tfrac{1}{\langle{w.\Lambda_{j},\lambda}\rangle}
    \Big\rbrace^{N}\,
    \Big\lbrace
    \chi_{w.S_{j}}(\lambda)\,
    \chi_{t}^{\infty}(\lambda)\,
    |\mathbf{c}(\lambda)|^{-2}\,
    e^{i\langle\lambda,A\rangle}
    \Big\rbrace.
\end{align*}
As in the the proof of \cref{S3.1 Kernel1}, we assume that
\begin{itemize}[leftmargin=20pt]
    \item 
    $N_{0}$ derivatives are applied to the cut-off function
    $\chi_{t}^{\infty}(\lambda)$:
    \begin{align*}
        \partial_{w.\Lambda_{j}}^{N_{0}}\,
        \chi_{t}^{\infty}(\lambda)\,
        =\,\mathrm{O}(|t|^{\frac{N_{0}}{2}}),
    \end{align*}
    
    \item
    $N_{1}$ derivatives are applied to the factors 
    $\tfrac{1}{\langle{w.\Lambda_{j},\lambda}\rangle}$,
    which produces $\mathrm{O}(|\lambda|^{-N-N_{1}})$,
    
    \item\vspace{5pt}
    $N_{2}$ derivatives are applied to the factor
    $\chi_{w.S_{j}}(\lambda)$, which is a homogeneous
    symbol of order $0$:
    \begin{align*}
        \partial_{w.\Lambda_{j}}^{N_{2}}\,
        \chi_{w.S_{j}}(\lambda)\,
        =\,\mathrm{O}(|\lambda|^{-N_{2}}),
    \end{align*}
    
    \item
    $N_{3}$ derivatives are applied to the factor
    $e^{i\langle\lambda,A\rangle}$:
    \begin{align*}
        \partial_{w.\Lambda_{j}}^{N_{3}}\,
        e^{i\langle\lambda,A\rangle}\,
        =\,\mathrm{O}(|A|^{N_{3}}),
    \end{align*}
    
    \item 
    $N_{4}$ derivatives are applied to the Plancherel density
    $|\mathbf{c}(\lambda)|^{-2}$ (see \cref{S2 dichotomy coro}):
    \begin{align*}
    \partial_{w.\Lambda_{j}}^{N_{4}}\,
    |\mathbf{c}(\lambda)|^{-2}\,
    =\,\mathrm{O}(|\lambda|^{d-\ell-N_{4}}),
    \end{align*}
\end{itemize}
with $N_{0}+N_{1}+N_{2}+N_{3}+N_{4}=N$.
Therefore, if no derivative hits the cut-off function
$\chi_{t}^{\infty}(\lambda)$, i.e., if $N_{0}=0$, then
\begin{align*}
    |I_{w.S_{j}}(t,x)|\,
    &\lesssim\,|t|^{-N}\,
    \int_{|\lambda|\gtrsim|t|^{-\frac{1}{2}}}\,d\lambda\,
    |\lambda|^{-N-N_{1}-N_{2}+d-\ell-N_{4}}\,
    |A|^{N_{3}}\\[5pt]
    &\lesssim\,|t|^{-N}\,
    |t|^{-\frac{d}{2}+\frac{N}{2}+\frac{N_{1}}{2}
    +\frac{N_{2}}{2}+\frac{N_{4}}{2}}\,
    |t|^{\frac{N_{3}}{2}}\,
    \le\,|t|^{-\frac{d}{2}}
\end{align*}
provided that $N>d$. 
If $N_{0}\ge1$, then the integral
is reduced to a spherical shell where
$|\lambda|\asymp|t|^{-\frac{1}{2}}$, and hence
\begin{align*}
    |I_{w.S_{j}}(t,x)|\,
    \lesssim\,|t|^{-N}\,|t|^{\frac{N_{0}}{2}}\,
    |t|^{-\frac{d}{2}+\frac{N}{2}+\frac{N_{1}}{2}
    +\frac{N_{2}}{2}+\frac{N_{4}}{2}}\,
    |A|^{N_{3}}\,
    \le\,|t|^{-\frac{d}{2}},
\end{align*}
since $|A|\le|x|\le\sqrt{|t|}$.
Together with \eqref{S3.2 Estimate I0},
we conclude that $|I(t,A)|\lesssim|t|^{-\frac{d}{2}}$
and
\begin{align*}
    |s_{t}(x)|\,
    \lesssim\,
        |t|^{-\frac{d}{2}}\,\varphi_{0}(x),
\end{align*}
for all $0<|t|<1$ and $x\in\mathbb{X}$ 
such that $|x|\le\sqrt{|t|}$.
\end{proof}

The above proof shows that, for every $\lambda\in
(\supp\chi_{w.S_{j}})\cap(\supp\chi_{t}^{\infty})$,
the Plancherel density $|\mathbf{c}(\lambda)|^{-2}$
behaves like an inhomogeneous symbol of order $d-\ell$
if we differentiate it along the direction $w.\Lambda_{j}$.
When $|x|>\sqrt{|t|}$, 
we write the Schrödinger propagator in terms of the wave propagator according to the subordination principle, and use an improved Hadamard parametrix.
Let us express the Schr\"odinger propagator
\begin{align*}
    e^{-it\mathbf{D}^{2}}\,
    =\,\underbrace{\pi^{-\frac{1}{2}}\,
        e^{-i\frac{\pi}{4}\sign{(t)}}}_{C_1}
        |t|^{-\frac{1}{2}}\,
        \int_{0}^{+\infty}ds\,
        e^{\frac{i}{4t}s^{2}}\cos(s\mathbf{D})
\end{align*}
in terms of the wave propagator and correspondingly
\begin{align}\label{S3.2 Kernel (wave)}
    s_{t}(x)\,
    =\,C_1\,|t|^{-\frac{1}{2}}\,
        \int_{0}^{+\infty}ds\,
        e^{\frac{i}{4t}s^{2}}\,\Phi_{s}(x)
\end{align}
for their bi-$K$-invariant convolution kernels.
On the one hand, by finite propagation speed,
\begin{align}\label{S3.2 Propagation speed}
    \Phi_{s}(x)\,=\,0
    \qquad\textnormal{if}\quad |x|>|s|.
\end{align}
On the other hand, recall the Hadamard parametrix
\begin{align}\label{S3.2 Parametrix}
    \Phi_{s}(\exp{H})\,
    =\,J(H)^{-\frac{1}{2}}\,|s|\,
    \sum_{k=0}^{+\infty}\,4^{-k}\,U_{k}(H)\,
    R_{+}^{k-\frac{d-1}{2}}(s^{2}-|H|^2)
    \qquad\forall\,s\in\mathbb{R}^{*},\,
    \forall\,H\in\mathfrak{p},
\end{align}
where $J$ denotes the Jacobian of the exponential map
$\mathfrak{p}\rightarrow{G/K}$, which is given by
\begin{align*}
\textstyle
    J(H)\,
    =\,\prod_{\alpha\in\Sigma^{+}}\,
    \big(
    \frac{\sinh\langle\alpha,H\rangle}{\langle\alpha,H\rangle}
    \big)^{m_{\alpha}}
    \qquad\forall\,H\in\mathfrak{a}^{+},
\end{align*}
and $\{R_{+}^{z}\,|\,z\in\mathbb{C}\}$ denotes the analytic family 
of Riesz distributions on $\mathbb{R}$, which is defined by
\begin{align*}
    \,R_{+}^{z}(r)\,=\,
    \begin{cases}
    \,\Gamma (z)^{-1}\,r^{z-1}
    &\textnormal{if \,}r>0\\
    \;0&\textnormal{if \,}r\le0
    \end{cases}
    \qquad\forall\,\re z>0.
\end{align*}
This parametrix was constructed and used in various 
settings, see for instance \cite{Ber1977,Hor1994,CGM2001}.
We refer to \cite[Appendix B]{AnZh2020} for the details about 
the wave propagator $\cos{(t\sqrt{-\Delta})}$ 
associated to the unshifted Laplacian $\Delta$ 
on non-compact symmetric spaces.
Notice that the same results hold for $\cos{(t\mathbf{D})}$.
Specifically \eqref{S3.2 Parametrix} is an 
asymptotic expansion
\begin{align}\label{S3.2 Expansion}
    \Phi_{s}(\exp{H})\,
    =\,J(H)^{-\frac{1}{2}}\,|s|\,
    \sum_{k=0}^{[d/2]}\,4^{-k}\,U_{k}(H)\,
    R_{+}^{k-\frac{d-1}{2}}(s^{2}-|H|^2)\,
    +\,E(s,H)
\end{align}
where the coefficients $U_{k}$ are $\Ad{K}$-invariant
smooth functions on $\mathfrak{p}$, which are bounded
together with their derivatives, while the remainder 
satisfies
\begin{align}\label{S3.2 Expansion remainder}
    |E(s,H)|\,
    \lesssim\,(1+|s|)^{3(\frac{d}{2}+1)}\,
    e^{-\langle\rho,H\rangle}
    \qquad\forall{s}\in\mathbb{R}^{*},\,
    \forall{H}\in\overline{\mathfrak{a}^{+}}.
\end{align}

Let us split up
\begin{align*}
    \int_{0}^{+\infty}\,ds\,
    =\,
    \int_{0}^{+\infty}ds\,\chi_{0}(\tfrac{s}{|x|})\,
    +\,\int_{0}^{+\infty}ds\,\chi_{1}(\tfrac{s}{|x|})\,
    +\,\int_{0}^{+\infty}ds\,\chi_{\infty}(\tfrac{s}{|x|})
\end{align*}
in \eqref{S3.2 Kernel (wave)} by means of a smooth
partition of unity $1=\chi_{0}+\chi_{1}+\chi_{\infty}$
on $\mathbb{R}$ such that
\begin{align*}
    \begin{cases}
    \supp\chi_{0}\,\subset\,(-1,1),\\[5pt]
    \supp\chi_{1}\,\subset\,
        (-2C_2,-\frac12)\,\cup\,(\frac12,2C_2),\\[5pt]
    \supp\chi_{\infty}\,\subset\,
        (-\infty,-C_2)\,\cup\,(C_2,+\infty)
    \end{cases}
\end{align*}
where the choice of $C_2>1$ will be specified later.
Then the contribution of the first integral vanishes
according to \eqref{S3.2 Propagation speed} and we are
left with
\begin{align*}
    s_{t}(x)\,
    =\,
    \underbrace{
        C_1\,|t|^{-\frac{1}{2}}\,
        \int_{0}^{+\infty}ds\,
        \chi_{1}(\tfrac{s}{|x|})\,
        e^{\frac{i}{4t}s^{2}}\,\Phi_{s}(x)
    }_{s_{t}^{1}(x)}
    \,+\,
    \underbrace{
        C_1\,|t|^{-\frac{1}{2}}\,
        \int_{0}^{+\infty}ds\,
        \chi_{\infty}(\tfrac{s}{|x|})\,
        e^{\frac{i}{4t}s^{2}}\,\Phi_{s}(x)
    }_{s_{t}^{\infty}(x)}
\end{align*}
where $s_{t}^{1}(x)$ and $s_{t}^{\infty}(x)$
are bi-$K$-invariant.
Let us first study $s_{t}^{\infty}(x)$ by using again
the barycentric decomposition. In comparison with the
proof of \cref{S3.2 Kernel2}, we have now $|x|>\sqrt{|t|}$ 
and there is an additional integral over $s\in(1,\infty)$ 
to control. Let us state the theorem.
\begin{theorem}
The following estimate holds, for all $0<|t|<1$ 
and $|x|>\sqrt{|t|}:$
\begin{align*}
|s_{t}^{\infty}(x)|\,\lesssim\,|t|^{-\frac{d}{2}}\,\varphi_{0}(x).
\end{align*}
\end{theorem}

\begin{proof}
We express
\begin{align*}
    s_{t}^{\infty}(x)\,
    =\, \tfrac{1}{2}\,C_0\,C_1\,
        |t|^{-\frac{1}{2}}\,
        \int_{-\infty}^{+\infty}ds\,
        \chi_{\infty}(\tfrac{s}{|x|})\,
        e^{\frac{i}{4t}s^{2}}\,
        \int_{\mathfrak{a}}\,d\lambda\,
        |\mathbf{c}(\lambda)|^{-2}\,
        \varphi_{\lambda}(x)\,
        e^{-is|\lambda|}
\end{align*}
by evenness and by expressing the wave kernel $\Phi_{s}$
by means of the inverse spherical Fourier transform.
Let us split up
$s_{t}^{\infty}
=\tfrac{1}{2}\,C_0\,C_1\,
(s_{t}^{\infty,0}+s_{t}^{\infty,\infty})$,
where
\begin{align*}
    s_{t}^{\infty,0}(x)\,
    =\,|t|^{-\frac12}
        \int_{-\infty}^{+\infty}ds\,
        \chi_{\infty}(\tfrac{s}{|x|})\,
        e^{\frac{i}{4t}s^{2}}\,
        \underbrace{
        \int_{\mathfrak{a}}\,d\lambda\,
        \chi_{t}^{0}(\lambda)\,
        |\mathbf{c}(\lambda)|^{-2}\,
        \varphi_{\lambda}(x)\,
        e^{-is|\lambda|}
        }_{I_{0}(s,t,x)}
\end{align*}
and
\begin{align*}
    s_{t}^{\infty,\infty}(x)\,
    =\,|t|^{-\frac12}
        \int_{-\infty}^{+\infty}ds\,
        \chi_{\infty}(\tfrac{s}{|x|})\,
        e^{\frac{i}{4t}s^{2}}\,
        \underbrace{
        \int_{\mathfrak{a}}\,d\lambda\,
        \chi_{t}^{\infty}(\lambda)\,
        |\mathbf{c}(\lambda)|^{-2}\,
        \varphi_{\lambda}(x)\,
        e^{-is|\lambda|}
        }_{I_{\infty}(s,t,x)}.
\end{align*}
Recall that 
$\chi_{t}^{0}(\lambda)=\chi(\sqrt{|t|}|\lambda|)$
is a radial cut-off function such that
$\supp\chi_{t}^{0}\subset\,B(0,2|t|^{-\frac12})$,
$\chi_{t}^{0}=1$ on $B(0,|t|^{-\frac12})$
and $\chi_{t}^{\infty}=1-\chi_{t}^{0}$.
\vspace{10pt}

\noindent\textbf{Estimate of $s_{t}^{\infty,0}$.}
Notice that the obvious estimate
$|I_{0}(s,t,x)|\lesssim|t|^{-\frac{d}{2}}|
\varphi_{0}(x)$
is not enough for our purpose. We need indeed to 
compensate on the one hand the factor $|t|^{-\frac12}$
and to get on the other hand enough decay in $|s|$
to ensure the convergence of the external integral.
To this end, we perform two integrations by parts based on 
\begin{align}\label{S3.2 Phase s}
    e^{\frac{i}{4t}s^{2}}\,
    =\,-\tfrac{2it}{s}\,
        \tfrac{\partial}{\partial{s}}\,
        e^{\frac{i}{4t}s^{2}}
\end{align}
and obtain this way
\begin{align}\label{S3.2 S0}
    s_{t}^{\infty,0}(x)\,
    =\,-4\,|t|^{\frac32}\,
        \int_{-\infty}^{+\infty}\,ds\,
        e^{\frac{i}{4t}s^{2}}\,
        \tfrac{\partial}{\partial{s}}
        \big(
        \tfrac{1}{s}\,\tfrac{\partial}{\partial{s}}
        \big)\,
        \big\lbrace
        \tfrac{1}{s}\,\chi_{\infty}(\tfrac{s}{|x|})\,
        I_{0}(s,t,x)\big\rbrace.
\end{align}
Notice that
\begin{align*}
    \big|(\tfrac{\partial}{\partial{s}})^{k}\,
    I_{0}(s,t,x)\big|\,
    \lesssim\,|t|^{-\frac{d+k}{2}}\,\varphi_{0}(x)
    \qquad\forall\,k\in\mathbb{N}.
\end{align*}
If any derivative hits $\chi_{\infty}(\frac{s}{|x|})$ in \eqref{S3.2 S0},
the integral reduces to two intervals where $|s|\asymp|x|$,
and the corresponding contribution is estimated by
\begin{align*}
    |t|^{-\frac{d-3}{2}}\,\varphi_{0}(x)\,
    \int_{|s|\asymp|x|}\,ds\,
    \lbrace
    s^{-2}\,|x|^{-2}\,+\,s^{-3}\,|x|^{-1}\,
    +\,s^{-2}\,|x|^{-1}\,|t|^{-\frac12}
    \rbrace\,
    \lesssim\,|t|^{-\frac{d}{2}}\,\varphi_{0}(x),
\end{align*}
since $|t|^{\frac12}<|x|$.
Otherwise we end up with the estimate
\begin{align*}
     |t|^{-\frac{d-3}{2}}\,\varphi_{0}(x)\,
    \int_{|s|\gtrsim|x|}\,ds\,
    \lbrace
    s^{-4}\,+\,s^{-3}\,|t|^{-\frac12}\,+\,s^{-2}\,|t|^{-1}
    \rbrace\,
    \lesssim\,|t|^{-\frac{d}{2}}\,\varphi_{0}(x).
\end{align*}
In conclusion,
\begin{align*}
|s_{t}^{\infty,0}(x)|\,\lesssim\,|t|^{-\frac{d}{2}}\,\varphi_{0}(x),
\end{align*}
for all $0<|t|<1$ and $x\in\mathbb{X}$ such that
$|x|>\sqrt{|t|}$.
\vspace{10pt}

\noindent\textbf{Estimate of $s_{t}^{\infty,\infty}$.}
Let us turn to
\begin{align*}
    s_{t}^{\infty,\infty}(x)\,
    =\,\int_{-\infty}^{+\infty}ds\,
        \chi_{\infty}(\tfrac{s}{|x|})\,
        e^{\frac{i}{4t}s^{2}}\,{I_{\infty}(s,t,x)}.
\end{align*}
We will prove the following estimate, for any integer $N>d$,
\begin{align}\label{S3.2 Estimate I infty s}
    |I_{\infty}(s,t,x)|
    \lesssim\,|s|^{-N}\,
        |t|^{-\frac{d}{2}+\frac{N}{2}}\,
        \varphi_{0}(x)
    \qquad\forall\,|s|\ge C_2|x|.
\end{align}
Then, as $|x|>\sqrt{|t|\,}$, we conclude easily that
\begin{align}\label{S3.2 S Infinity}
    |s_{t}^{\infty,\infty}(x)|\,
    \lesssim\,
        |t|^{-\frac{d}{2}+\frac{N}{2}-\frac12}\,
        \varphi_{0}(x)\,
        \int_{|s|\gtrsim|x|}ds\,|s|^{-N}
    \lesssim\,|t|^{-\frac{d}{2}}\,
        \big(\tfrac{\sqrt{|t|}}{|x|}\big)^{N-1}\,
        \varphi_{0}(x)\,
    \lesssim\,|t|^{-\frac{d}{2}}\,\varphi_{0}(x).
\end{align}
In order to establish \eqref{S3.2 Estimate I infty s},
we express
\begin{align*}
    I_{\infty}(s,t,x)\,
    &=\,\int_{\mathfrak{a}}\,d\lambda\,
        \chi_{t}^{\infty}(\lambda)\,
        |\mathbf{c}(\lambda)|^{-2}\,
        \varphi_{\lambda}(x)\,
        e^{-is|\lambda|}\\[5pt]
    &=\,\int_{K}dk\,e^{-\langle\rho,A\rangle}
        \sum_{w\in W}\,\sum_{1\le j\le\ell}\,
        \underbrace{
        \int_{\mathfrak{a}}\,d\lambda\,
        \chi_{w.S_{j}}(\lambda)\,
        \chi_{t}^{\infty}(\lambda)\,
        |\mathbf{c}(\lambda)|^{-2}\,
        e^{-i(s|\lambda|-\langle\lambda,A\rangle)}
        }_{I_{w.S_{j}}(s,t,A)}
\end{align*}
by using again the integral formula \eqref{S2 Spherical Function}
and the barycentric decomposition \eqref{S2 Partition}.
According to \eqref{S2 barycentric 1}, we can choose $C_2>0$
such that, if $|s|\ge{C_2|x|}$, then
\begin{align*}
    |\partial_{w.\Lambda_{j}}
    (s|\lambda|-\langle\lambda,A\rangle)|\,
    &=\,\Big|s\tfrac{\langle{w.\Lambda_{j},\lambda}\rangle}{|\lambda|}\,
    -\,\langle{w.\Lambda_{j},A}\rangle\Big|\\[5pt]
    &\ge\,|s|\,
    \underbrace{
    \tfrac{|\langle{w.\Lambda_{j},\lambda}\rangle|}{|\lambda|}
    }_{\gtrsim1}
    -\,
    \underbrace{|\langle{w.\Lambda_{j},A}\rangle|}_{
    \lesssim\,|x|}
    \gtrsim\,|s|,
\end{align*}
for every
$\lambda\in(\supp\chi_{t}^{\infty})\cap(\supp\chi_{w.S_{j}})$.
Under these assumptions, the phase function 
$\lambda\mapsto\,s|\lambda|-\langle\lambda,A\rangle$
has no critical point along the direction $w.\Lambda_{j}$.
By performing $N$ integrations by parts based on
\begin{align*}
    e^{-i(s|\lambda|-\langle\lambda,A\rangle)}\,
    =\,\tfrac{i}{\partial_{w.\Lambda_{j}}
    (s|\lambda|-\langle\lambda,A\rangle)}\;
    \partial_{w.\Lambda_{j}}\,
    e^{-i(s|\lambda|-\langle\lambda,A\rangle)},
\end{align*}
we write
\begin{align*}
    I_{w.S_{j}}(s,t,A)\,
    =\,(is)^{-N}\,
    \int_{\mathfrak{a}}\,&d\lambda\,
    e^{-i(s|\lambda|-\langle\lambda,A\rangle)}\\[5pt]
    &\Big\lbrace\partial_{w.\Lambda_{j}}\circ
    \tfrac{s}{\partial_{w.\Lambda_{j}}
    (s|\lambda|-\langle\lambda,A\rangle)}\Big\rbrace^{N}\,
    \Big\lbrace \chi_{w.S_{j}}(\lambda)\,
    \chi_{t}^{\infty}(\lambda)\,
    |\mathbf{c}(\lambda)|^{-2}
    \Big\rbrace.
\end{align*}
Assume that
\vspace{2pt}
\begin{itemize}[leftmargin=20pt]
\item 
$N_{0}$ derivatives are applied to the cut-off function $\chi_{t}^{\infty}(\lambda)$,
which produces $\mathrm{O}(|t|^{\frac{N_0}2})$,
\vspace{5pt}
\item
$N_{1}$ derivatives are applied to the factors
$\tfrac{s}{\partial_{w.\Lambda_{j}}(s|\lambda|-\langle\lambda,A\rangle)}$,
which produces $\mathrm{O}(|\lambda|^{-N_1})$,
\item
$N_{2}$ derivatives are applied to the cut-off functions $\chi_{w.S_{j}}(\lambda)$,
which produces $\mathrm{O}(|\lambda|^{-N_2})$,
\vspace{4pt}
\item 
$N_{3}$ derivatives are applied to the Plancherel density $|\mathbf{c}(\lambda)|^{-2}$,
which produces 
$\mathrm{O}(|\lambda|^{d-\ell-N_3})$,
\end{itemize}
\vspace{2pt}
with $N_{0}+N_{1}+N_{2}+N_{3}=N$.
Again, if some derivatives hit $\chi_{t}^{\infty}(\lambda)$, i.e., if $N_0\ge1$,
then the integral reduces to a spherical shell where $|\lambda|\asymp\,|t|^{-\frac12}$,
and its contribution is estimated by
\begin{align*}
    |s|^{-N}\,
    |t|^{-\frac{\ell}{2}}\,|t|^{\frac{N_{0}}{2}}\,
    |t|^{\frac{N_{1}}{2}}\,|t|^{\frac{N_{2}}{2}}\,
    |t|^{-\frac{d}{2}+\frac{\ell}{2}+\frac{N_{3}}{2}}\,
    =\,|s|^{-N}\,|t|^{-\frac{d}{2}+\frac{N}{2}}.
\end{align*}
If $N_{0}=0$, then
\begin{align*}
    |I_{w.S_{j}}(s,t,A)|\,
    \lesssim\,|s|^{-N}
    \int_{|\lambda|\gtrsim\,|t|^{-\frac12}}\,d\lambda\,
    |\lambda|^{-N_{1}}\,|\lambda|^{-N_{2}}\,
    |\lambda|^{d-\ell-N_{3}}\,
    \lesssim\,|s|^{-N}\,|t|^{-\frac{d}{2}+\frac{N}{2}}
\end{align*}
provided that $N>d$.
This proves \eqref{S3.2 Estimate I infty s} and hence
\eqref{S3.2 S Infinity}.
\end{proof}

\begin{theorem}\label{S3.2 Kernel3}
The following estimate holds, for all $0<|t|<1$ and $|x|>\sqrt{|t|}$:
\begin{align*}
    |s_{t}^{1}(x)|\,
    \lesssim\,
        |t|^{-\frac{d}{2}}\,
        (1+|x|)^{\frac{3}{2}d+4}\,
        e^{-\langle\rho,x^{+}\rangle}.
\end{align*}
\end{theorem}

\begin{proof}
Since $s_{t}^{1}$ is bi-$K$-invariant, we have
\begin{align*}
    s_{t}^{1}(x)\,
    &=\,
    \tfrac{C_1}{2}\,|t|^{-\frac{1}{2}}\,
    J(x^{+})^{-\frac{1}{2}}\,
    \sum_{k=0}^{[d/2]}\,4^{-k}\,U_{k}(x^{+})\,
    \overbrace{
    \int_{0}^{+\infty}d(s^{2})\,
    \chi_{1}(\tfrac{s}{|x|})\,
        e^{\frac{i}{4t}s^{2}}\,
        R_{+}^{k-\frac{d-1}{2}}(s^{2}\!-|x|^2)
    }^{I_{k}(t,|x|)}\\[5pt]
    &\qquad+\,
    \tfrac{C_1}{2}\,|t|^{-\frac{1}{2}}\,
    \underbrace{
    \int_{0}^{+\infty}ds\,
    \chi_{1}(\tfrac{s}{|x|})\,
    e^{\frac{i}{4t}s^{2}}\,E(s,x^{+})
    }_{\widetilde{E}(t,|x|)}
\end{align*}
according to \eqref{S3.2 Expansion}.
On the one hand, the remainder estimate
\begin{align}\label{S3.2 Remainder estimate}
    |\widetilde{E}(t,|x|)|\,
    \lesssim\,|x|\,(1+|x|)^{3(\frac{d}{2}+1)}\,
        e^{-\langle\rho,x^{+}\rangle}
\end{align}
follows from \eqref{S3.2 Expansion remainder}.
On the other hand, we claim that
\begin{align}\label{S3.2 Ik estimate}
    |I_{k}(t,|x|)|\,\lesssim\,|t|^{k-\frac{d-1}2}
\end{align}
if $|x|\!>\!\sqrt{|t|\,}\!>\!0$.
Let us first prove \eqref{S3.2 Ik estimate} 
when $d$ is odd. By a change of variables and  by using the fact that 
$R_{+}^{k-\frac{d-1}{2}}(s\!-\!1)
=\bigl(\tfrac{\partial}{\partial{s}}\bigr)^{\!\frac{d-1}{2}-k}R_{+}^{0}(s\!-\!1)$,
we obtain
\begin{align*}
|I_{k}(t,|x|)|\,
&=\,|x|^{2k-d+1}\int_{0}^{+\infty}ds\,\chi_{1}(\sqrt{s})\,e^{\frac{i|x|^{2}}{4t}s}\,R_{+}^{k-\frac{d-1}{2}}(s-1)\\
&=\,|x|^{2k-d+1}\int_{0}^{+\infty}ds\,R_{+}^{0}(s\!-\!1)\,\bigl(-\tfrac\partial{\partial s}\bigr)^{\!\frac{d-1}2-k}\,
\big|\lbrace{\chi_{1}(\sqrt{s})\,e^{\frac{i|x|^{2}}{4t}s}}\bigr\rbrace.
\end{align*}
As $R_{+}^{0}(s\!-\!1)$ is the Dirac measure at $s\!=\!1$,
we conclude that
\begin{align*}
I_k(t,|x|)\,
=\,|x|^{2k-d+1}\,\bigl(-\tfrac i4\tfrac{|x|^2}t\bigr)^{\!\frac{d-1}2-k}\,
=\,\mathrm{O}(|t|^{k-\frac{d-1}2}).
\end{align*}
When $d$ is even, we obtain similarly
\begin{align*}
|I_{k}(t,|x|)|\,
=\,\pi^{-\frac12}\,|x|^{2k-d+1}
\int_1^{+\infty}\tfrac{ds}{\sqrt{s-1}}\,
\bigl(-\tfrac{\partial}{\partial{s}}\bigr)^{\frac{d}{2}-k}\,
\bigl\lbrace{\chi_{1}(\sqrt{s})\,e^{\frac{i|x|^{2}}{4t}s}}\bigr\rbrace,
\end{align*}
which is a linear combination of expressions
\begin{align*}
t^{j+k-\frac{d}{2}}\,|x|^{1-2j}
\underbrace{\int_1^{+\infty}\,\tfrac{ds}{\sqrt{s-1}}\,\theta_j(s)\,e^{\frac{i|x|^2}{4t}s}}_{J_{j}(t,|x|)}
\end{align*}
where $0\le{j}\le\frac{d}{2}-k$ and
$\theta_{j}\in\mathcal{C}_{c}^{\infty}{(\mathbb{R})}$
with $\supp\theta_{j}\subset(-4C_2^2,4C_2^2)$.
Notice that the elementary estimate $J_{j}(t,|x|)=\mathrm{O}(1)$,
together with the assumption $|x|\!>\!\sqrt{|t|\,}$ implies that
$$
|I_{k}(t,|x|)|\lesssim|x|\,|t|^{-\frac d2+k},
$$
which might be enough for our purpose as long as $k\!>\!0$.
The case $k\!=\!0$ requires actually a more careful analysis.
Let us show that 
\begin{align*}
    J_{j}(t,|x|)\,\lesssim\,\tfrac{\sqrt{|t|}}{|x|}
\end{align*}
by splitting up
\begin{align*}
\int_1^{+\infty}ds\,
=\,\int_{1}^{1+\tfrac{|t|}{|x|^{2}}}ds
\,+\,\int_{1+\tfrac{|t|}{|x|^{2}}}^{+\infty}ds
\end{align*}
in the definition of $J_{j}(t,|x|)$.
The contribution of the first integral is easily  estimated by
\begin{align*}
\int_1^{\,1+\tfrac{|t|}{|x|^{2}}}ds\,\tfrac{ds}{\sqrt{s-1}}\,
=\,2\,\sqrt{s\!-\!1}\,\Big|_{s=1}^{s=1+\tfrac{|t|}{|x|^{2}}}
=\,2\,\tfrac{\sqrt{|t|}}{|x|}\,.
\end{align*}
After performing an integration by parts based on
\begin{align*}
e^{\frac{i|x|^{2}}{4t}s}\,
=\,-\,i\,\tfrac{4t}{|x|^2}\,\tfrac{\partial}{\partial{s}}\,e^{\frac{i|x|^{2}}{4t}s},
\end{align*}
the contribution of the second integral is also estimated by
\begin{align*}
\tfrac{|t|}{|x|^2}
\int_{1+\tfrac{|t|}{|x|^2}}^{\,4C_2^2}ds\,\bigl\{(s\!-\!1)^{-\frac12}+(s\!-\!1)^{-\frac32}\bigr\}\,
\lesssim\,\tfrac{\sqrt{|t|}}{|x|}
\end{align*}
under the assumption $|x|\!>\!\sqrt{|t|\,}$.
Thus \eqref{S3.2 Ik estimate} holds as well
when $d$ is even. In conclusion,
\begin{align*}
    |s_{t}^{1}(x)|\,
    &\lesssim\,
        |t|^{-\frac{d}{2}}\,J(x^{+})^{-\frac12}
        \,+\,
        |t|^{-\frac12}\,
        |x|\,(1+|x|)^{3(\frac{d}{2}+1)}\,
        e^{-\langle\rho,x^+\rangle}\\[5pt]
    &\lesssim\,
        |t|^{-\frac{d}{2}}\,
        (1+|x|)^{\frac{3}{2}d+4}
        e^{-\langle\rho,x^+\rangle}
\end{align*}
when $0<|t|<1$ and $x\in\mathbb{X}$ satisfies $|x|>\sqrt{|t|\,}$.
\end{proof}
\vspace{5pt}

In summary, we have divided our kernel analysis into three parts
and deduced \cref{S1.Main Thm Kernel} from \cref{S3.1 Kernel1}, 
\cref{S3.2 Kernel2} and \cref{S3.2 Kernel3}.
Notice that the method used to prove small time kernel 
estimates can be also used for large time.

\section{Dispersive estimates and Strichartz inequalities on symmetric spaces}\label{Section.4 Dispersive Strichartz}
Once pointwise kernel estimates are available on symmetric spaces,
one can deduce dispersive properties for the corresponding propagator, 
by using an interpolation argument based on the Kunze-Stein phenomenon.
The following bi-$K$-invariant version is a straightforward generalization of \cite[Theorem 4.2]{APV2011}.
\begin{lemma}\label{Kunze-Stein}
Let $\kappa$ be a reasonable bi-$K$-invariant function on $G$. Then
\begin{align*}
    \|\cdot*\,\kappa\,\|_{L^{q'}(\mathbb{X})
                    \rightarrow
                    L^{q}(\mathbb{X})}
    \le \Big\lbrace \int_{G} dx\, 
        \varphi_{0}(x)\,
        |\kappa(x)|^{\frac{q}{2}}\, 
        \Big\rbrace^{\frac{2}{q}}
\end{align*}
for every $q\in[2,+\infty)$. In the limit case, we have
\begin{align*}
\textstyle
    \|\cdot*\,\kappa\,\|_{L^{1}(\mathbb{X})
                    \rightarrow
                    L^{\infty}(\mathbb{X})}
    = \sup_{x\in{G}} |\kappa(x)|.
\end{align*}
\end{lemma}

\begin{proof}[Proof of \cref{S1.Main Thm Dispersive}]
At the endpoint $q=\widetilde{q}=2$, $t\mapsto\,e^{it\Delta}$
is a one-parameter group of unitary operators on $L^2(\mathbb{X})$:
\begin{align*}
    \|e^{it\Delta}\|_{L^{2}(\mathbb{X})
                    \rightarrow
                    L^{2}(\mathbb{X})}\,
    =\,1.
\end{align*}
For $2<q<\infty$, we deduce from estimates \eqref{S1 kernel estimates} and \eqref{S2 density delta} that
\begin{align*}
    \|s_{t}\|_{L^{q}(\mathbb{X})}\,
    \lesssim\,\omega(t)\,\Big\lbrace{
        \int_{\mathfrak{a}}\,dx^{+}\,
        (1+|x^{+}|)^{qN}\,
        e^{-(q-2)\langle{\rho,x^{+}}\rangle}
        }\Big\rbrace^{\frac{1}{q}}\,\lesssim\,\omega(t)
\end{align*}
where
\begin{align*}
    \omega(t)\,=\,
    \begin{cases}
        |t|^{-\frac{d}{2}} 
        &\textnormal{if \,}0<|t|<1,\\[5pt]
        |t|^{-\frac{D}{2}} 
        &\textnormal{if \,}|t|\ge1.
    \end{cases}
\end{align*}
Moreover, since $s_{t}$ is the bi-$K$-invariant convolution kernel of the propagator $e^{it\Delta}$,
we deduce, by combining the above lemma with \eqref{S1 kernel estimates}, \eqref{S2 density delta}, and \eqref{S2 phi0} that
\begin{align*}
    \|e^{it\Delta}\|_{L^{q'}(\mathbb{X})
                    \rightarrow
                    L^{q}(\mathbb{X})}^{\frac{q}{2}}\,
    &\le\,\int_{\mathfrak{a}}\,dx^{+}\,\delta(x^{+})\,\varphi_{0}(\exp{x^{+}})\,
        |s_{t}(\exp{x^{+}})|^{\frac{q}{2}}\\[5pt]
    &\lesssim\,
        \omega(t)^{\frac{q}{2}}\,
        \underbrace{\Big\lbrace{
        \int_{\mathfrak{a}}\,dx^{+}\,
        (1+|x^{+}|)^{\frac{qN+D-\ell}{2}}\,
        e^{-(\frac{q}{2}-1)\langle{\rho,x^{+}}\rangle}
        }\Big\rbrace}_{=\,\mathrm{O}(1)}.
\end{align*}
We conclude by interpolation between
\begin{align*}
    \begin{cases}
        \|e^{it\Delta}\|_{L^{1}(\mathbb{X})\rightarrow{L}^{q}(\mathbb{X})}\,
        \le\,\|s_{t}\|_{L^{q}(\mathbb{X})}\,\lesssim\,|t|^{-\frac{d}{2}},\\[5pt]
        \|e^{it\Delta}\|_{L^{q'}(\mathbb{X})\rightarrow{L}^{\infty}(\mathbb{X})}\,
        \le\,\|s_{t}\|_{L^{q}(\mathbb{X})}\,\lesssim\,|t|^{-\frac{d}{2}},\\[5pt]
        \|e^{it\Delta}\|_{L^{2}(\mathbb{X})\rightarrow{L}^{2}(\mathbb{X})}\,=\,1,
    \end{cases}
\end{align*}
when $|t|\le1$ and 
\begin{align*}
    \begin{cases}
        \|e^{it\Delta}\|_{L^{1}(\mathbb{X})\rightarrow{L}^{q}(\mathbb{X})}\,
        \le\,\|s_{t}\|_{L^{q}(\mathbb{X})}\,\lesssim\,|t|^{-\frac{D}{2}},\\[5pt]
        \|e^{it\Delta}\|_{L^{q'}(\mathbb{X})\rightarrow{L}^{\infty}(\mathbb{X})}\,
        \le\,\|s_{t}\|_{L^{q}(\mathbb{X})}\,\lesssim\,|t|^{-\frac{D}{2}},\\[5pt]
        \|e^{it\Delta}\|_{L^{q'}(\mathbb{X})\rightarrow{L}^{q}(\mathbb{X})}\,
        \lesssim\,|t|^{-\frac{D}{2}},
    \end{cases}
\end{align*}
when $|t|\ge1$.
\end{proof}

Finally, we establish our Strichartz inequality \eqref{S1 Strichart inequality} by using the $TT^{*}$ argument. Notice that the solution to the Schrödinger equation \eqref{S1 Schrodinger F} is given by the Duhamel formula:
\begin{align*}
    u(t,x)\,
    =\,e^{it\Delta}\,f(x)\,
        -\,i\int_{0}^{t}\,ds\,
        e^{i(t-s)\Delta}\,F(s,x).
\end{align*}
Consider the operator $T$ and its formal adjoint $T^{*}$:
\begin{align*}
    Tf(t,x)\,=\,e^{it\Delta}f(x)
    \qquad\textnormal{and}\qquad
    T^{*}F(x)\,=\,\int_{\mathbb{R}}\,ds\,e^{-is\Delta}\,F(s,x).
\end{align*}
By duality, $T$ is bounded from $L_{x}^{2}$ to $L_{t}^{p}L_{x}^{q}$ if and only if 
$T^{*}$ is bounded from $L_{t}^{p'}L_{x}^{q'}$ to $L_{x}^{2}$. 
Equivalently, the operator
\begin{align}\label{S4 TT Operator}
    TT^{*}F(t,x)\,=\,
    \int_{\mathbb{R}}\,ds\,e^{i(t-s)\Delta}\,F(s,x)
\end{align}
is bounded from $L_{t}^{p'}L_{x}^{q'}$ to $L_{t}^{p}L_{x}^{q}$.
We prove the latter for all $(p,q)$ in the admissible triangle \eqref{S1 admissible X}. The Strichartz inequality \eqref{S1 Strichart inequality} follows from the Christ-Kiselev lemma, see \cite{ChKi2001} or \cite[Section 2.3]{Tao2006}.

\begin{proof}
The endpoint $(\frac{1}{p},\frac{1}{q})=(0,\frac{1}{2})$ is settled by $L^{2}$ conservation and $(\frac{1}{p},\frac{1}{q})=(\frac{1}{2},\frac{1}{2}-\frac{1}{d})$ can be handled by using the method carried out in \cite{KeTa1998}. Let us focus on the non-endpoint cases where $(\frac{1}{2}-\frac{1}{q})\frac{n}{2}\le\frac{1}{p}\le\frac{1}{2}$ and $\frac{1}{2}-\frac{1}{n}<\frac{1}{q}<\frac{1}{2}$.
According to the dispersive estimates \eqref{S1.Main Dispersive properties}, we have
\begin{align*}
    \|TT^{*}F\|_{L^{p}(I;L^{q}(\mathbb{X}))}\,
    &\le\,
    \Big\|\int_{\mathbb{R}}\,ds\,
    \big\|e^{i(t-s)\Delta}\,F(s,\,.\,)\big\|_{L^{q}(\mathbb{X})}\Big\|_{L_{t}^{p}(I)}\\[5pt]
    &\lesssim\,
    \Big\|\int_{|t-s|\le1}ds\,|t-s|^{-(\frac{1}{2}-\frac{1}{q})d}\,
    \big\|F(s,\,.\,)\big\|_{L^{q'}(\mathbb{X})}\Big\|_{L_{t}^{p}(I)}\\[5pt]
    &\quad+\,
    \Big\|\int_{|t-s|\ge1}ds\,|t-s|^{-\frac{D}{2}}\,
    \big\|F(s,\,.\,)\big\|_{L^{q'}(\mathbb{X})}\Big\|_{L_{t}^{p}(I)}.
\end{align*}
On the one hand, the convolution kernel $|t-s|^{-\frac{D}{2}}\mathbbm{1}_{\lbrace|t-s|\ge1\rbrace}$ on $\mathbb{R}$ defines a bounded operator from $L^{p'}$ to $L^{p}$ for all $2\le{p}\le\infty$. On the other hand,  the convolution kernel $|t-s|^{-(\frac{1}{2}-\frac{1}{q})d}\mathbbm{1}_{\lbrace|t-s|\le1\rbrace}$ is bounded from $L^{p'}$ to $L^{p}$ for all $2\le{p}<\infty$ such that $0\le\frac{1}{p'}-\frac{1}{p}\le1-(\frac{1}{2}-\frac{1}{q})d$, or in other words, when $p$ satisfies $(\frac{1}{2}-\frac{1}{q})\frac{n}{2}\le\frac{1}{p}\le\frac{1}{2}$.
\end{proof}

\begin{remark}
We do not take full advantage of our strong large time decay $|t|^{-\frac{D}{2}}$ in the dispersive estimates \eqref{S1.Main Dispersive properties}. Indeed, the convolution kernel $|t-s|^{-1-\varepsilon}\mathbbm{1}_{\lbrace|t-s|\ge1\rbrace}$ defines a bounded operator from $L^{p'}$ to $L^{p}$ for any $\varepsilon>0$ and $2\le{p}\le\infty$.
\end{remark}

\noindent\textbf{Acknowledgments.}
The last author acknowledges financial support from the Methusalem Programme 
    \textit{Analysis and Partial Differential Equations (Grant number 01M01021)} 
    during his postdoc stay at Ghent University.


\vspace{20pt}
\printbibliography
\vspace{20pt}

\address{
\noindent\textsc{Jean-Philippe Anker:}
\href{mailto:anker@univ-orleans.fr}
{anker@univ-orleans.fr}\\
Institut Denis Poisson,
Universit\'e d'Orl\'eans, Universit\'e de Tours \& CNRS,
Orl\'eans, France}
\vspace{10pt}

\address{
\noindent\textsc{Stefano Meda:}
\href{stefano.meda@unimib.it}
{stefano.meda@unimib.it}\\
Dipartimento di Matematica e Applicazioni,
Universit\`a di Milano-Bicocca,
Milano, Italia}
\vspace{10pt}

\address{
\noindent\textsc{Vittoria Pierfelice:}
\href{mailto:vittoria.pierfelice@univ-orleans.fr}
{vittoria.pierfelice@univ-orleans.fr}\\
Institut Denis Poisson,
Universit\'e d'Orl\'eans, Universit\'e de Tours \& CNRS,
Orl\'eans, France}
\vspace{10pt}

\address{
\noindent\textsc{Maria Vallarino:}
\href{mailto: maria.vallarino@polito.it}
{maria.vallarino@polito.it}\\
Dipartimento di Scienze Matematiche Giuseppe Luigi
Lagrange, Dipartimento di Eccellenza 2018-2022,
Politecnico di Torino,
Torino, Italia}
\vspace{10pt}

\address{
\noindent\textsc{Hong-Wei Zhang:}
\href{mailto:hongwei.zhang@ugent.be}
{hongwei.zhang@ugent.be}\\
Institut Denis Poisson,
Universit\'e d'Orl\'eans, Universit\'e de Tours \& CNRS,
Orl\'eans, France\\
Department of Mathematics,
Ghent University,
Ghent, Belgium}
\end{document}